\documentclass[11pt,reqno, a4paper]{amsart}
\usepackage{amsmath,amsfonts,amsthm,amsopn,color,amssymb,enumitem}
\usepackage{palatino}
\usepackage{graphicx}
\usepackage{caption}
\usepackage{subcaption}
\usepackage[colorlinks=true]{hyperref}
\usepackage{geometry}
\hypersetup{urlcolor=blue, citecolor=red, linkcolor=blue}

\usepackage{cite}
\usepackage{esint}
\usepackage{pgfplots}
\usepackage{mathrsfs}
\usetikzlibrary{arrows}
\usepackage{verbatim}
\usepackage{mathtools}

\newtheorem{theorem}{Theorem}[section]
\newtheorem{lemma}[theorem]{Lemma}

\newtheorem{proposition}[theorem]{Proposition}
\newtheorem{remark}[theorem]{Remark}

\newtheorem{corollary}[theorem]{Corollary}

\newcommand{\R}{\mathbb{R}}

\newcommand{\N}{\mathbb{N}}

\def\bbm[#1]{\mbox{\boldmath $#1$}}
\newcommand{\beq }{\begin{equation}}
\newcommand{\eeq }{\end{equation}}

\def\sideremark#1{\ifvmode\leavevmode\fi\vadjust{\vbox to0pt{\vss
 \hbox to 0pt{\hskip\hsize\hskip1em
 \vbox{\hsize3cm\tiny\raggedright\pretolerance10000
  \noindent #1\hfill}\hss}\vbox to8pt{\vfil}\vss}}}%

\pgfplotsset{compat=1.18}

\begin{document}

\title[Second-order boundary estimates]{Second-order boundary estimates for solutions to a class of quasilinear elliptic equations}

\author{Giuseppe Spadaro and Domenico Vuono}

\email[Giuseppe Spadaro]{giuseppe.spadaro@unical.it}
\email[Domenico Vuono]{domenico.vuono@unical.it}
\address[ G. Spadaro, D. Vuono]{Dipartimento di Matematica e Informatica, Università della Calabria,
Ponte Pietro Bucci 31B, 87036 Arcavacata di Rende, Cosenza, Italy}




\begin{abstract}
We prove global second-order regularity for a class of quasilinear elliptic equations, both with homogeneous Dirichlet and Neumann boundary conditions. A condition on the integrability of the second fundamental form on the boundary of the domain is required. As a consequence, with the additional assumption that the source term has a sign, we obtain integrability properties of the inverse of the gradient of the solution. 
Assuming convexity of the domain, no boundary regularity is required.
\end{abstract}

\maketitle

\section{Introduction}

This work focuses on establishing global second-order estimates for solutions to boundary value problems, either of Dirichlet type, namely:
\begin{equation}\label{Dir_pb}
\begin{cases}
-\operatorname{div} (a(|\nabla u|) \nabla u )= f(x) &\text{in } \Omega\\
u = 0 &\text{on } \partial \Omega,
\end{cases}
\end{equation}
or of Neumann type:
\begin{equation}\label{Neu_pb}
\begin{cases}
-\operatorname{div} (a(|\nabla u|) \nabla u )= f(x) &\text{in } \Omega\\
\frac{\partial u}{\partial \nu} = 0 &\text{on } \partial \Omega,
\end{cases}
\end{equation}
with the compatibility condition:
\begin{equation*}
    \int_{\Omega} f \ dx = 0.
\end{equation*}
Here $\Omega$ is a bounded open set in $\mathbb{R}^n$ and $f\in W^{1,1}(\Omega) \cap L(\Omega)$, where:
\begin{equation}\label{L_Omega}
    L(\Omega) = \begin{cases}
        L^{n,1}(\Omega) & \text{if }\  n \geq 3,\\
        L^q(\Omega) \text{ with } q > n & \text{if }\  n = 2.    \end{cases}
\end{equation}
Moreover, $\nu$ denotes the outward unit normal on $\partial \Omega$ and the function $a:(0,\infty)\rightarrow (0,\infty)$ is of class $C^1(0,\infty)$, and such that:
\begin{equation}\label{cond_a}
-1 < i_a : = \inf_{t>0} \frac{ta'(t)}{a(t)} \leq \sup_{t>0} \frac{ta'(t)}{a(t)} = : s_a < \infty.
\end{equation}
In particular, choosing $a(t) = t^{p-2}$ with $p > 1$ we recover the classical $p$-Laplace operator, which fits into this framework since $i_a=s_a=p-2$ in this case.\\
Notice that \eqref{Dir_pb} and \eqref{Neu_pb} correspond to the Euler–Lagrange equations derived from the strictly convex energy functional:
\begin{equation*}
    J(u):=\int_{\Omega} B(|\nabla u|) - fu \, dx,
\end{equation*}
in the proper Orlicz-Sobolev spaces, where the function $B:[0,\infty)\rightarrow[0,\infty)$ is given by:
\begin{equation}\label{def_B}
    B(t) = \int_0^t b(\tau) \ d\tau \quad, t\geq 0,
\end{equation}
and $b:[0,\infty) \rightarrow [0,\infty)$ is defined as:
\begin{equation}\label{b}
    b(t) = a(t)t.
\end{equation}
Since the function $b$, by \eqref{cond_a}, is strictly increasing for $t \geq 0$, then the function $B$ is strictly convex.

The central aim of this work is to prove optimal regularity of the stress field $a(|\nabla u|)\nabla u$ up to the boundary. This problem has been studied by Cianchi and Maz’ya in \cite{CiaMa}, where the authors proved, under minimal assumptions on the domain $\Omega$, that $a(|\nabla u|)\nabla u \in W^{1,2}(\Omega)$ if and only if $f\in L^2(\Omega)$, under either Dirichlet or Neumann boundary conditions.

Our focus is on establishing stronger second-order regularity of the solutions, which requires more stringent regularity assumptions on the source term and on the domain.

Here, we require a condition on the summability of the second fundamental form on $\partial \Omega$. Namely, we assume that $\Omega$ is a bounded Lipschitz domain and that the functions locally describing the boundary of $\Omega$ possess second-order weak derivatives belonging to a specific Lorentz–Zygmund space, depending on the dimension $n$.
\begin{theorem}\label{teo1INTRO}
	Let $\Omega$ be a bounded Lipschitz domain in $\R^n$, with $\partial \Omega\in W^2X$, where

\begin{equation}\label{X}
    X=
    \begin{cases}
        L^{n-1,1} & \mbox{if $n \geq 3$,}\\
        L\log{L} & \mbox{if $n =2$.}
    \end{cases}
\end{equation}
Assume that the function $a \in C^1(0,\infty)$ satisfies \eqref{cond_a}. Let ${u}$ be a weak solution to either the Neumann problem \eqref{Neu_pb} or the Dirichlet problem \eqref{Dir_pb}, with $$ f(x)\in W^{1,1}(\Omega)\cap L(\Omega).$$
Then,
\begin{equation}\label{ciao2INTRO}
    a(|\nabla u|)^k\ \nabla u \in W^{1,2}(\Omega),
\end{equation}
with
\begin{equation}\label{T_k}
    \begin{cases}
        k > \frac{s_a- 1}{2s_a}, &\text{if }\  0 \leq i_a < s_a,\\
        k < \frac{i_a-1}{2i_a},&\text{if }\  i_a < s_a\leq0,\\
        \frac{s_a-1}{2s_a} < k < \frac{i_a-1}{2i_a},&\text{if }\  i_a < 0 < s_a.\\
    \end{cases}
\end{equation}
\end{theorem}
\vspace{0.1cm}
\begin{remark}
    The established regularity is optimal, as counterexamples show that stronger regularity results cannot be expected. Indeed, in \cite[Remark $1.4$]{SMM_vec} the authors show that, in the case of the $p$-Laplacian (i.e., $a(t) = t^{p-2}$), the solution does not exhibit more regularity than what we have proved in Theorem \ref{teo1INTRO}, even assuming a smooth right-hand side.\\
    However, we cannot claim that the result is optimal with respect to the regularity assumed on the source term.
\end{remark}
\begin{remark}
    Notice that with $k = 1$ we recover the result in \cite{CiaMa}, namely $a(|\nabla u|)\nabla u \in W^{1,2}(\Omega)$.
\end{remark}
\vspace{0.1cm}
As a corollary of Theorem \ref{teo1INTRO}, we deduce the following Sobolev regularity for the second-order derivatives of solutions of both the Dirichlet and Neumann boundary value problems.
\begin{corollary}\label{CorW22}
Let $\Omega$, $a$, $f$ and $u$ as in Theorem \ref{teo1INTRO}. Suppose that $\inf_{t \in[0,M]} a(t) = 0$, for every $M>0$, then:
\begin{equation}\label{W22}
    u\in W^{2,2}(\Omega) \quad \text{if } \ s_a < 1.
\end{equation}
If instead, $\inf_{t \in [0,M]} a(t) > 0$, for every $M>0$, then \eqref{W22} holds regardless of $s_a$.
\end{corollary}
\begin{remark}
    In \cite{CiaMa}, the authors proved that the solutions of \eqref{Dir_pb} and \eqref{Neu_pb} are actually in $W^{2,2}(\Omega)$, under the condition 
    \begin{equation*}
        \inf_{t \in [0,M]} a(t) > 0,
    \end{equation*}
    for every $M>0$, and assuming stronger regularity on the source term and the domain.
\end{remark}
Global second-order estimates in a basic form hold for the solutions to \eqref{Dir_pb} and \eqref{Neu_pb} in any bounded convex domain $\Omega$:
\begin{theorem}[Convex domains]\label{conv_d}
    Let $\Omega$ be a bounded convex open set in $\mathbb{R}^n$, with $n \geq 2$. Assume that the function $a \in C^1(0,\infty)$ obeys \eqref{cond_a}.
    Let ${u}$ be a weak solution to either the Neumann problem \eqref{Neu_pb} or the Dirichlet problem \eqref{Dir_pb}, with $$ f(x)\in W^{1,1}(\Omega)\cap L(\Omega).$$ Then,
    \begin{equation*}
        a(|\nabla u|)^k\ \nabla u \in W^{1,2}(\Omega),
    \end{equation*}
    with $k$ as in Theorem \ref{teo1INTRO}.\\
Moreover, either if $\inf_{t \in [0,M]} a(t) > 0$, for every $M>0$, or if $s_a < 1$, we have that $u\in W^{2,2}(\Omega)$.
\end{theorem}
The fact that this result is valid without any regularity assumptions on the domain $\Omega$ is a consequence of the semidefinite nature of the second fundamental form on the boundary of a convex set.
We emphasize that Theorem \ref{teo1INTRO} and Theorem \ref{conv_d} hold without any sign assumption on the source term $f$. However, if a sign condition on $f$ is additionally imposed, we can further derive integrability properties of the inverse of the gradient of the solution, namely:
\begin{theorem}\label{peso_stima_Intro}
     Let $u$, $a$, $\Omega$ and $f$ as in Theorem \ref{teo1INTRO}. Let $x\in \overline \Omega$. Suppose that 
     \begin{equation*}
         f \geq \tau_x > 0\quad \text{or } f \leq -\tau_x < 0 \quad \text{in } B_\rho(x)\cap \overline\Omega,
     \end{equation*}
  for some $\rho$ positive constant. 
Let $\beta\in \R$, such that 

\begin{equation}\label{casi_di_beta}
    \beta\in 
    \begin{cases}
        \left(-\infty,\frac{s_a+1}{s_a}\right) & \mbox{if $0\leq i_a<s_a$,}\vspace{0.4cm}\\ \vspace{0.4cm}
        \left(\frac{i_a+1}{i_a},\frac{s_a+1}{s_a}\right) & \mbox{if $i_a< 0<s_a$,}\\ 
        \left(\frac{i_a+1}{i_a},+\infty\right) & \mbox{if $i_a<s_a \leq 0$.}
    \end{cases}
\end{equation}

Then we have:
     \begin{equation*}
         \int_{\Omega\cap B_\rho(x)} \frac{1}{{(a(|\nabla  {u}|}))^\beta}\leq C,
     \end{equation*}
where $C=C(i_a,s_a,\beta,\tau,\rho,n,\Omega,\|{f}\|_{L(\Omega)},\|{f}\|_{W^{1,1}(\Omega)})$ is a positive constant. \\
     \noindent As a consequence, if $\inf_{t\in[0,M]} a(t)=0$, for every $M>0$, and $s_a\geq 1$, then $$u\in W^{2,q}(\Omega),\quad  \text{with } 1\leq q< \frac{s_a+1}{s_a}.$$
\end{theorem}
As before, a basic version of Theorem \ref{peso_stima_Intro} holds for convex domains.
\begin{theorem}[Convex domains]\label{conv_d_peso}
    Let $\Omega$ be a bounded convex open set in $\mathbb{R}^n$. Let $\alpha$, ${u}$, $a$, ${f}$ and $\beta$ as in Theorem \ref{peso_stima_Intro}. Then \begin{equation*}
         \int_{\Omega\cap B_\rho(x)} \frac{1}{{(a(|\nabla  {u}|}))^\beta}\leq C.
     \end{equation*}
    Moreover, if $\inf_{t\in[0,M]} a(t)=0$, for every $M>0$, and $s_a\geq 1$, then $$u\in W^{2,q}(\Omega),\quad  \text{with } 1\leq q< \frac{s_a+1}{s_a}.$$
\end{theorem}
 Before starting the proofs of our results let us discuss the state of the art on the regularity theory for quasilinear elliptic equations. It is well known that solutions to \eqref{Dir_pb} and \eqref{Neu_pb} are not classical and in general, under suitable assumptions on the domain $\Omega$ and on the source term $f$, are of class $C^{1,\alpha}(\overline \Omega)$, see \cite{Lib0,DB,tolk}.

 The issue of studying second-order regularity of solutions to quasilinear elliptic equations has been the subject of extensive research.
 As mentioned above, if $f\in L^2(\Omega)$ the $W^{1,2}$ regularity of the stress field has been established in \cite{CiaMa}. Similar results, in this direction, are available in \cite{antonuovo,Sarsa,lou}.\\
 If $f$ is more regular, stronger second-order estimates can be obtained. In particular, in \cite{SMM_vec}, the authors proved sharp second-order estimates up to the boundary for the stress field $|\nabla u|^{\alpha-1}\nabla u$, with $\alpha > (p-1)/{2}$, of the $p$-Laplace equation, namely proving a global version of the local results in \cite{DamSci}. Note that the result in \cite{SMM_vec} was obtained for $C^3$-smooth domains via a fine argument based on the Fermi coordinates. As a corollary, the authors proved that the solution $u\in W^{2,2}(\Omega)$, for $1<p<3$. Via similar techniques, analogous results has been proved for more general elliptic operators, see \cite{Ce,CMS,S1,S2,EST}.

The anisotropic counterpart of the aforementioned results, which presents additional analytical challenges due to the lack of isotropy in the underlying structure, has been addressed in \cite{Anto1, ACCFM, BMV, CaRiSc}.\\
Stronger regularity results of solutions to $p$-Laplace equation, are obtained, for $p$ close to $2$, in \cite{BSV,FV,MRS}.

 Additionally, let us mention some results concerning second-order estimates in the vectorial case, as developed in \cite{M,MMSV,BaCiDiMa,Cma}. Here, following the ideas developed in \cite{SSV}, we obtain the counterpart, for a more general class of quasilinear elliptic equations, of the results proved in that paper.\\
To conclude, we mention several contributions within the framework of regularity theory, see \cite{AvKuMi,Bar,DeFMi,DM,DM2,DM3,DM4,DB,Dong,12,13,GuMo,KuuMin,16,23,24}. 

\textbf{Organization.} The paper is organized as follows. In Section \ref{S2_O}, we recall some well-known notions about Orlicz-Sobolev spaces. In Section \ref{S3_O}, we state some preliminary results. In Section \ref{S4_O}, we prove a global integral inequality needed in order to prove our main results. In Section \ref{second-order} and Section \ref{integrability_section}, we give the proofs of Theorem \ref{teo1INTRO} and Theorem \ref{peso_stima_Intro}.
\section{Function spaces and notion of weak solutions}\label{S2_O}
In this section, we recall the definitions and some properties of Orlicz and Orlicz-Sobolev spaces. In light of these definitions we give the notion of weak solutions of problems \eqref{Dir_pb} and \eqref{Neu_pb}. We refer to \cite{Ci2,Ci3,Ci4,Cma_S} for further details.
\subsection{Orlicz-Sobolev spaces}
We denote by $B:[0,\infty) \rightarrow [0,\infty]$ a Young function, i.e. a convex function such that $B(0)=0$. Moreover, if $0 < B(t) <\infty$ for $t>0$, 
and
\begin{equation*}
    \lim_{t \rightarrow 0} \frac{B(t)}{t}=0 \quad \text{and } \quad \lim_{t \rightarrow \infty} \frac{B(t)}{t}=\infty,
\end{equation*}
then $B$ is called an $N$-function.\\
The Young conjugate $\tilde B$ of a Young function $B$, is the Young function defined as:
\begin{equation*}
    \tilde B (t) = \sup \{ tr - B(r): r \geq 0\}.
\end{equation*}
Notice that if $B$ is an $N$-function, also $\tilde B$ is an $N$-function.\\
A Young function $B$ is said to belong to the class $\Delta_2$ if there exists a constant $C > 1$ such that:
\begin{equation*}
    B(2t) \leq CB(t) \quad \text{for } t >0.
\end{equation*}
Let $\Omega$ be an open bounded set in $\mathbb{R}^n$. The Orlicz space $L^B(\Omega)$ is the Banach space of functions $u : \Omega \rightarrow \mathbb{R}$ such that:
\begin{equation*}
    \|u\|_{L^B(\Omega)} = \inf \left \{  \lambda > 0 : \int_{\Omega} B\left( \frac{|u(x)|}{\lambda}\right) \ dx \leq 1\right \},
\end{equation*}
is finite.\\

The Orlicz-Sobolev space $W^{1,B}(\Omega)$ is the Banach space defined as:
\begin{equation*}
    W^{1,B}(\Omega) = \{ u  \in L^{B}(\Omega) : \text{is weakly differentiable in } \Omega \text{ and } |\nabla u| \in L^{B}(\Omega) \},
\end{equation*}
with associated norm:
\begin{equation*}
    \|u\|_{W^{1,B}(\Omega)} = \|u\|_{L^B(\Omega)}+\|\nabla u\|_{L^B(\Omega)}.
\end{equation*}
Moreover, we set $W^{1,B}_0(\Omega)$ as the Banach subspace of $W^{1,B}(\Omega)$ given by:
\begin{equation*}
\begin{split}
    W^{1,B}_0(\Omega)& = \{ u  \in W^{1,B}(\Omega) : \\
    &\text{the continuation of $u$ by $0$ outside $\Omega$ is weakly differentiable in } \mathbb{R}^n \}.
\end{split}
\end{equation*}
The Banach subspace $W_{\perp}^{1,B}(\Omega)$  of $W^{1,B}(\Omega)$ is defined as:
\begin{equation*}
    W_{\perp}^{1,B}(\Omega) = \{u \in W^{1,B}(\Omega) : u_\Omega=0\},
\end{equation*}
where
\begin{equation*}
    u_\Omega = \frac{1}{|\Omega|}\int_{\Omega} u(x)dx.
\end{equation*}
Let any young function $B$, such that
\begin{equation}\label{consizioneB}
   \int_0 \left( \frac{t}{B(t)} \right)^{\frac{1}{n-1}} \ dt < \infty.
\end{equation}
The Sobolev conjugate of $B$ is the Young function $B_n$ defined as:
\begin{equation*}
    B_n(t) = B(H_n^{-1}(t)) \quad \text{for } t \geq 0,
\end{equation*}
where:
\begin{equation*}
    H_n(s) = \left (\int_0^s \left( \frac{t}{B(t)} \right)^{\frac{1}{n-1}} \ dt\right)^{\frac{1}{n'}} \quad \text{for } s \geq 0,
\end{equation*}
and $H_n^{-1}$ denotes the left continuous inverse of $H_n$.
If $\Omega$ is bounded, a Sobolev-type inequality holds for every Young function $B$, and for every $u\in W^{1,B}_0(\Omega)$, namely:
\begin{equation}\label{Sobolev}
    \|u\|_{L^{B_n}(\Omega)} \leq C \|\nabla u\|_{L^B(\Omega)},
\end{equation}
with $C=C(n,|\Omega|)$. If, in addition, $\Omega$ has a Lipschitz boundary the same inequality holds for every function $u \in W_{\perp}^{1,B}(\Omega)$.\\
In particular, since it is quite easy to see that:
\begin{equation*}
    L^{B_n}(\Omega) \rightarrow L^B(\Omega),
\end{equation*}
we have:
\begin{equation}\label{Poincare}
    \|u\|_{L^{B}(\Omega)} \leq C \|\nabla u\|_{L^B(\Omega)},
\end{equation}
with $C=C(n,|\Omega|,B)$. Moreover the embedding is compact. Notice that condition \eqref{consizioneB} is irrelevant by embeddings of Orlicz spaces.\\
Let us now state two well-known theorems, the first one related to the density of smooth functions into Orlicz-Sobolev spaces; the second one recalls the reflexivity property of these spaces, see also \cite{DT}.
\begin{theorem}\label{imbedding}
    Let $\Omega$ be a bounded open set in $\mathbb{R}^n$, $n \geq 2$. Assume that $B$ is a Young function such that $B \in \Delta_2$. Then, the space $C_0^\infty(\Omega)$ is dense in $W_0^{1,B}(\Omega)$.\\
    If, in addition, $\Omega$ has a Lipschitz boundary, then $C^\infty(\overline \Omega)$ is dense in $W^{1,B}(\Omega)$.
\end{theorem}
\begin{theorem}\label{reflexivity}
    Let $\Omega$ be a bounded open set in $\mathbb{R}^n$, $n \geq 2$. Let $B$ be a Young function such that $B\in \Delta_2$ and $\tilde B \in \Delta_2$. Then, the spaces $W^{1,B}(\Omega)$, $W^{1,B}_0(\Omega)$ and $W^{1,B}_{\perp}(\Omega)$ are reflexive.
\end{theorem}

\subsection{Weak solutions.}
First of all, we recall the following (see \cite{Cma_S}).
\begin{proposition}
    Assume that $a \in C^1(0,\infty)$ and fulfills \eqref{cond_a}. Let $B$ be the function defined in \eqref{def_B}. Then, $B$ is a strictly convex $N$-function, and
    \begin{equation}\label{delta_cond}
        B \in \Delta_2 \quad \text{and } \quad \tilde B \in \Delta_2.
    \end{equation}
    Moreover, there exists a constant $C(i_a,s_a)$ such that
    \begin{equation}\label{bound_left}
        \tilde B(b(t)) \leq CB(t) \quad \text{for } t \geq 0.
    \end{equation}
\end{proposition}
\vspace{0.2cm}
Let $\Omega$ be a bounded open set in $\mathbb{R}^n$ and $f \in L^{\tilde{B_n}}(\Omega)$. A weak solution to \eqref{Dir_pb} is a function $u \in W_0^{1,B}(\Omega)$ such that:
\begin{equation}\label{weak_dir}
    \int_{\Omega} a(|\nabla u|)\nabla u \cdot \nabla \varphi = \int_{\Omega} f \varphi \ dx,
\end{equation}
for every $\varphi \in W_0^{1,B}(\Omega)$. Here, the dot “$\cdot$” denotes the standard scalar product.\\
Assume also that $\Omega$ is a Lipschitz domain. A weak solution to \eqref{Neu_pb} is a function $u \in W^{1,B}(\Omega)$ such that:
\begin{equation}\label{weak_neu}
    \int_{\Omega} a(|\nabla u|)\nabla u \cdot \nabla \varphi = \int_{\Omega} f \varphi \ dx,
\end{equation}
for every $\varphi \in W^{1,B}(\Omega)$.

\begin{remark}
    By \eqref{bound_left}, the left hand sides of \eqref{weak_dir} and \eqref{weak_neu} are well defined; moreover, by \eqref{Sobolev} the right hand sides are also well defined.\\
    We also point out that by \eqref{delta_cond}, equivalent formulations can be established using test functions in the spaces $C_0^\infty(\Omega)$ and $C^\infty(\overline \Omega)$, in view of Theorem \ref{imbedding}.
\end{remark}
\section{Preliminary results}\label{S3_O}
This section is devoted to the statement of several technical lemmas, which are essential for the proofs of the main results.\\
We begin with a lemma, see \cite{Cma_S} for details, in which any function $a$ as in Theorem \ref{teo1INTRO}-\ref{conv_d_peso}  is approximated by a family ${a_\varepsilon}$ of smooth functions whose indices $i_{a_\varepsilon}$ and $s_{a_\varepsilon}$ can be estimated in terms of $i_a$ and $s_a$, respectively.
\begin{lemma}\cite[Lemma $3.3$]{Cma_S}\label{lemma_a_eps}
Let $a:(0,\infty)\rightarrow(0,\infty)$ be $C^1(0,\infty)$ and that fulfills \eqref{cond_a}. Then there exists a family of functions $\{a_\varepsilon\}_{\varepsilon \in (0,1)} \subset C^\infty([0,\infty))$ such that:
\begin{equation*}
    a_\varepsilon : [0,\infty) \rightarrow (0,\infty),
\end{equation*}

\begin{equation*}
    \min\{i_a,0\}\leq i_{a_\varepsilon} \leq s_{a_\varepsilon} \leq \max\{s_a,0\},
\end{equation*}

\begin{equation*}
    \lim_{\varepsilon \rightarrow 0} b_\varepsilon = b \quad \text{uniformly in } [0,M] \text{ for every } M>0,
\end{equation*}
and hence,
\begin{equation*}
    \lim_{\varepsilon \rightarrow 0} B_\varepsilon = B \quad \text{uniformly in } [0,M] \text{ for every } M>0,
\end{equation*}
where $b_\varepsilon$ and $B_\varepsilon$ are defined, respectively, as in \eqref{b} and \eqref{def_B}, with $a$ replaced by $a_\varepsilon$.\\
Moreover, 
\begin{equation}\label{unif_conv}
    \lim_{\varepsilon \rightarrow 0} a_{\varepsilon}(|\xi|)\xi = a(|\xi|)\xi \quad \text{ uniformly in } \quad \{\xi \in \mathbb{R}^n : |\xi| \leq M\}, \ \forall M  > 0.
\end{equation}
\end{lemma}
\begin{remark}
Notice that, from a close inspection of the proof of Lemma \ref{lemma_a_eps}, $a_\varepsilon$ is chosen in the following way:
\begin{equation}\label{a_cap}
    a_\varepsilon(t) = \hat a_\varepsilon((\varepsilon+t^{2})^\frac{1}{2}),
\end{equation}
for some function $\hat a_\varepsilon \in C^{\infty}([0,\infty))$ and such that $a_\varepsilon(t)>0$ if $t>0$ and,
\begin{equation}\label{cond_a_hat}
 i_a \leq \inf_{t>0} \frac{t \hat a'_\varepsilon(t)}{\hat a_\varepsilon(t)} \leq \sup_{t>0} \frac{t \hat a'_\varepsilon(t)}{\hat a_\varepsilon(t)} \leq s_a.
\end{equation}
Moreover, we have that:
\begin{equation}\label{conv_a_eps}
    \lim_{\varepsilon \rightarrow 0} \hat a_\varepsilon = a \quad \text{ uniformly in [$L,M$] for every $M>L>0$},
\end{equation}
and, for $t\geq 0$,
\begin{equation}\label{estimate_a_eps}
    a(1)\min\{(\varepsilon + t^2)^\frac{i_{a}}{2},(\varepsilon + t^2)^\frac{s_{a}}{2}\} \leq a_{\varepsilon}(t) \leq a(1)\max\{(\varepsilon + t^2)^\frac{i_{a}}{2},(\varepsilon + t^2)^\frac{s_{a}}{2}\}.
\end{equation}
\end{remark}
\subsection{Setting of the problem}.
We consider $\Omega$ a bounded Lipschitz domain and we assume that the functions of $(n-1)$ variables that
locally describe the boundary of $\Omega$ are twice weakly differentiable with second derivatives in a suitable space $X$, briefly $\partial \Omega \in  W^2X$.
An open set $\Omega$ in $\R^n$ is called a Lipschitz domain if there exist constants $L_\Omega>0$ and $R_\Omega\in (0,1)$ such that, for every $x_0\in \partial \Omega$ and $R \in (0, R_\Omega]$ there exist an orthogonal coordinate system centered at $0\in \R^n$ and an $L_\Omega$-Lipschitz continuous function $\psi : B'_{R}\rightarrow (-l,l)$, where $B'_{R}$ denotes the ball in $R^{n-1}$, centered at $0' \in \R^{n-1}$ and with radius $R$, and $l=R(1+L_\Omega)$, satisfying 

\begin{equation*}
    \begin{split}
        \partial \Omega \cap \left(B'_R \times (-l,l)\right)=\{(x',\psi(x')) : x'\in B'_R\} \\ \Omega \cap \left(B'_R \times (-l,l)\right)=\{(x',x_n) : \psi(x')<x_n<l\}.
    \end{split}
\end{equation*}

 Moreover, we call Lipschitz characteristic of $\Omega$ the couple $\mathcal{L}_\Omega=(L_\Omega,R_\Omega)$. 

\

\

Let $\Omega$ be a Lipschitz domain with Lipschitz characteristic $\mathcal{L}_\Omega = (L_\Omega, R_\Omega)$, and let $\rho\in L^1(\partial \Omega)$ be a nonnegative function. We set, for $r \in(0, R_\Omega]$,

\begin{equation}\label{defK}
    \mathcal{K}_{\Omega,\rho}(r)=\sup_{E\subset B_r(x) ,  x\in\partial \Omega} \frac{\int_{\partial \Omega\cap E}\rho \,d\mathcal{H}^{n-1}}{cap(B_r(x),E)},
\end{equation}
and
\begin{equation}\label{integrability}
    \Psi_{\Omega,\rho}(r)=
    \begin{cases}
        \sup_{x \in \partial \Omega} \| \rho \|_{L^{n-1,\infty}(\partial \Omega \cap B_r(x))} & \mbox{if $n \geq 3$}\\
        \sup_{x \in \partial \Omega} \| \rho \|_{L^{1,\infty}\log{L}(\partial \Omega \cap B_r(x))} & \mbox{if $n = 2$}
    \end{cases}
\end{equation}

Here, $B_r(x)$ stands for the ball centered at $x$, with radius $r$, the notation $cap(B_r(x),E)$ denotes the capacity of the set $E$ relative to the ball $B_r(x)$, and $\mathcal{H}^{n-1}$ is the $(n-1)$-dimensional Hausdorff measure. Moreover, we denote with $d_\Omega$ the diameter of the domain $\Omega$.

Recall that the Lorentz spaces $L^{p,q}(\partial \Omega)$, with $p > 1$, are the Banach function spaces endowed with the norm:
\begin{equation*}
    \| f \|_{L^{p,q}(\partial \Omega)} = 
    \begin{cases}
        \left \{ \int_0^{\mathcal{H}^{n-1}(\partial\Omega)} [t^\frac{1}{p} f^{**}(t)]^q \frac{dt}{t}\right\}^\frac{1}{q}& \mbox{if $1 \leq q < \infty$}\\
        \sup_{0<t<\mathcal{H}^{n-1}(\partial\Omega)} \{t^\frac{1}{p} f^{**}(t)\}& \mbox{if $q=\infty$}
    \end{cases}
\end{equation*}
for a measurable function $f$ on $\partial \Omega$. Here, $f^{**}(t) = \frac{1}{t} \int_{0}^t f^*(s)ds$ for $t>0$, where $f^*$ stands for the decreasing rearrangement of $f$. The Lorentz-Zygmund space $L^{p,q}\log L(\partial \Omega)$, with $p\geq 1$, is equipped with the norm
\begin{equation*}
    \| f \|_{L^{p,q}\log L(\partial \Omega)} = 
    \begin{cases}
        \left \{ \int_0^{\mathcal{H}^{n-1}(\partial\Omega)} [t^\frac{1}{p} \log(1+\frac{1}{t}) f^{**}(t)]^q \frac{dt}{t}\right\}^\frac{1}{q}& \mbox{if $1 \leq q < \infty$}\\
        \sup_{0<t<\mathcal{H}^{n-1}(\partial\Omega)} \{t^\frac{1}{p} \log(1+\frac{1}{t}) f^{**}(t)\}& \mbox{if $q=\infty$}.
    \end{cases}
\end{equation*}

Moreover, we set 
\begin{equation}\label{KB}
    \mathcal{K}_{\Omega}(r)=\sup_{E\subset B_r(x) ,  x\in\partial \Omega} \frac{\int_{\partial \Omega\cap E}|\mathcal{B}| \,d\mathcal{H}^{n-1}}{cap(B_r(x),E)},
\end{equation}
and
\begin{equation}\label{integrabilityB}
    \Psi_{\Omega}(r)=\sup_{x \in \partial \Omega} \| \mathcal{B}\|_{\tilde X(\partial \Omega \cap B_r(x))}
\end{equation}
where $\mathcal{B}$ stands for the weak second fundamental form on $\partial \Omega$, $|\mathcal{B}|$ its norm and
\begin{equation*}
    \tilde X=
    \begin{cases}
        L^{n-1,\infty} & \mbox{if $n \geq 3$,}\\
        L^{1,\infty}\log{L} & \mbox{if $n =2$.}
    \end{cases}
\end{equation*}

\begin{remark}
    The integrability assumption $\partial \Omega \in W^2 X$ ensures that the functions that locally describe the boundary have continuous first-order derivatives, hence $\partial \Omega \in C^1$.
\end{remark}

\begin{remark}\label{antony2}
    Notice that if $\partial \Omega \in W^2 X$, where $X$ is defined as in \eqref{X}, then the following condition holds:
    \begin{equation}\label{cond_B}
        \lim_{r \rightarrow 0^+} \left (\sup_{x \in \partial \Omega} \| \mathcal{B} \|_{\tilde X(\partial \Omega \cap B_r(x))}\right) < c,
    \end{equation}
    for a suitable positive constant $c=c(n,N,p,L_\Omega,d_\Omega)$. Indeed,
    \begin{equation*}
        L^{n-1,1} \subset L^{n-1} \subset L^{n-1,\infty},
    \end{equation*}
    and
    \begin{equation*}
        L \log L \subset L^{1,\infty}\log L.
    \end{equation*}
    In particular, condition \eqref{cond_B} is fulfilled if $\partial \Omega \in C^2$.
\end{remark}

Now we state two lemmas which will be useful to us later, see \cite[Section 6]{ACCFM} and \cite{Cma,Ant}. 
\begin{lemma}\label{antoninho2}
    Let $\Omega$ be a bounded Lipschitz domain in $\R^n$, $n\geq 2$, with Lipschitz characteristic $\mathcal{L}_\Omega =(L_\Omega, R_\Omega)$. Assume that $\rho$ is a nonnegative function on $\partial \Omega$ such that  $\rho\in L^1(\partial \Omega)$. Then,
    \begin{equation}
        \mathcal{K}_{\Omega,\rho}(r) \leq c \Psi_{\Omega,\rho}(r),
    \end{equation}
    for some constant $c=c(n,L_\Omega)$, for every $r \in (0,R_\Omega]$.
\end{lemma}

\begin{lemma}\label{antoninho}
    Let $\Omega$ be a bounded Lipschitz domain in $\R^n$, $n\geq 2$, with Lipschitz characteristic $\mathcal{L}_\Omega =(L_\Omega, R_\Omega).$ Assume that $\rho$ is a nonnegative function on $\partial \Omega$ such that  $\rho\in L^1(\partial \Omega)$. Then, there exists a constant $C(n)$, such that  
    \begin{equation}\label{antony}
        \int_{\partial \Omega \cap B_r(x_0)}v^2\rho \,d\mathcal{H}^{n-1} \leq C(n)(1+L_\Omega)^4K_{\Omega,\rho}(r)\int_{\Omega\cap B_r(x_0)}|\nabla v|^2 \,dx
    \end{equation}
    for every $x_0\in \partial \Omega$, for every $r\in (0,R_\Omega]$ and for every $v \in W^{1,2}_0(B_r(x_0))$.
\end{lemma}
We conclude this section with the following result (see \cite[Lemma 3.3]{BaCiDiMa}):
\begin{lemma}\label{stima_punt}
    Let $\omega \in R^n$ be such that $|\omega|=1$. Then
    \begin{equation*}
        |H\omega|^2 - \frac{1}{2} |\omega \cdot H\omega|^2 - \frac{1}{2} |H|^2 \leq 0,
    \end{equation*}
    for every $H \in R^{n\times n}_{\operatorname{sym}}$.
\end{lemma}
\section{A global integral inequality}\label{S4_O}
In what follows we denote as $u_i$ the partial derivative of the function $u$ w.r.t. the variable $x_i$, namely $u_i = \partial_{x_i} u$.\\
A fundamental aspect of our approach is grounded in global integral inequalities formulated for smooth functions defined on smooth domains:
\begin{theorem}\label{STIMA}
Let $\Omega$ be a bounded open set in $\mathbb{R}^n$ with $C^2$-boundary and $\alpha \in [0,1)$. Let $\varepsilon \in (0,1)$ and $a_\varepsilon:[0,\infty) \rightarrow [0,\infty)$ be the function defined in \eqref{a_cap}.\\
There exists a constant $\overline c=\overline c(n,s_a,i_a,\alpha,L_\Omega)$ such that, if
    \begin{equation}\label{condKr}
         K_\Omega (r) \leq K(r), \quad \forall{} r \in (0,1),
    \end{equation}
    for some $K$ satisfying
    \begin{equation}\label{condizionesuk}
        \lim_{r\rightarrow 0^+} K (r) <\overline c.
    \end{equation}
 Then, for any smooth function $u \in C^3(\Omega) \cap C^2(\overline \Omega)$ and for some positive $C(n,s_a,i_a,\alpha,L_\Omega,d_\Omega,K_\Omega)$, either
\begin{equation}\label{Dir_estimate}
\begin{split}
&\int_\Omega \frac{a_\varepsilon(|\nabla u|)}{(\varepsilon + |\nabla u|^2)^{\frac{\alpha}{2}}} |D^2u|^2 \ dx\\
&\qquad \leq \mathcal{C} \left (\int_{\Omega} \frac{a_\varepsilon(|\nabla u|) }{(\varepsilon + |\nabla u|^2)^{\frac{\alpha}{2}}} |\nabla u|^2\ dx \right.\\
&\left.\qquad \qquad + \sum_{i=1}^n \int_\Omega \operatorname{div}(a_\varepsilon(|\nabla u|)\nabla u)  \partial_{x_i} \left (\frac{u_i}{(\varepsilon + |\nabla u|^2)^{\frac{\alpha}{2}}}\right)\ dx\right.\\
&\left. \qquad \qquad + \int_\Omega \operatorname{div}(a_\varepsilon(|\nabla u|)\nabla u) \frac{|\nabla u|}{(\varepsilon + |\nabla u|^2)^{\frac{\alpha}{2}}} \ dx\right),
\end{split}
\end{equation}
if $u=0$ on $\partial \Omega$, or
\begin{equation}\label{Neu_estimate_intro}
\begin{split}
&\int_\Omega \frac{a_\varepsilon(|\nabla u|)}{(\varepsilon + |\nabla u|^2)^{\frac{\alpha}{2}}} |D^2u|^2 \ dx\\
&\qquad \leq \mathcal{C} \left (\int_{\Omega} \frac{a_\varepsilon(|\nabla u|) }{(\varepsilon + |\nabla u|^2)^{\frac{\alpha}{2}}} |\nabla u|^2 \ dx \right.\\
&\left.\qquad \qquad + \sum_{i=1}^n \int_\Omega \operatorname{div}(a_\varepsilon(|\nabla u|)\nabla u) \partial_{x_i} \left(\frac{u_i}{(\varepsilon + |\nabla u|^2)^{\frac{\alpha}{2}}}\right)\ dx\right).
\end{split}
\end{equation}
if $\frac{\partial u}{\partial \nu}=0$ on $\partial \Omega$.\\
Moreover, if $\Omega$ is convex, inequalities \eqref{Dir_estimate} and \eqref{Neu_estimate_intro} hold regardless of $K_\Omega$.
\end{theorem}

\begin{proof}
Let $u \in C^3(\Omega) \cap C^2(\overline \Omega)$ and $\varphi \in C_c^\infty(\mathbb{R}^n)$. For $i=1,..,n$, computations lead to:
\begin{equation*}
\begin{split}
-\int_\Omega &\partial_{x_i} \operatorname{div}(a_\varepsilon(|\nabla u|)\nabla u) \frac{u_i \varphi^2}{(\varepsilon + |\nabla u|^2)^{\frac{\alpha}{2}}}\ dx\\
&= - \int_{\Omega}\operatorname{div}\left ( \partial_{x_i}(a_\varepsilon(|\nabla u|))\nabla u + a_\varepsilon(|\nabla u|) \nabla u_i \right ) \frac{u_i \varphi^2}{(\varepsilon + |\nabla u|^2)^{\frac{\alpha}{2}}}\ dx\\
&=-\int_\Omega \operatorname{div}\left ( \frac{\hat a'_\varepsilon((\varepsilon+|\nabla u|^{2})^\frac{1}{2})}{(\varepsilon + |\nabla u|^2)^\frac{1}{2}}(\nabla u \cdot \nabla u_i)\nabla u \right ) \frac{u_i \varphi^2}{(\varepsilon + |\nabla u|^2)^{\frac{\alpha}{2}}}\ dx\\
& \qquad - \int_{\Omega}\operatorname{div}\left (a_\varepsilon(|\nabla u|) \nabla u_i \right ) \frac{u_i \varphi^2}{(\varepsilon + |\nabla u|^2)^{\frac{\alpha}{2}}}\ dx\\
\end{split}
\end{equation*}
Since $u \in C^2(\overline \Omega)$, we are allowed to integrate by parts the right hand side of the previous equation, namely:
\begin{equation*}
\begin{split}
-\int_\Omega &\partial_{x_i} \operatorname{div}(a_\varepsilon(|\nabla u|)\nabla u) \frac{u_i \varphi^2}{(\varepsilon + |\nabla u|^2)^{\frac{\alpha}{2}}}\ dx\\
&=\int_\Omega\frac{\hat a'_\varepsilon((\varepsilon+|\nabla u|^{2})^\frac{1}{2})}{(\varepsilon + |\nabla u|^2)^\frac{1}{2}}(\nabla u \cdot \nabla u_i) \left (\nabla u \cdot \nabla \left (\frac{u_i \varphi^2}{(\varepsilon + |\nabla u|^2)^{\frac{\alpha}{2}}}\right ) \right )\ dx\\
&\qquad  - \int_{\partial \Omega} \frac{\hat a'_\varepsilon((\varepsilon+|\nabla u|^{2})^\frac{1}{2})}{(\varepsilon + |\nabla u|^2)^\frac{1}{2}}(\nabla u \cdot \nabla u_i) (\nabla u \cdot \nu) \frac{u_i \varphi^2}{(\varepsilon + |\nabla u|^2)^{\frac{\alpha}{2}}}\ d\mathcal{H}^{n-1}\\
& \qquad + \int_{\Omega} a_\varepsilon(|\nabla u|) \nabla u_i \cdot \nabla \left ( \frac{u_i \varphi^2}{(\varepsilon + |\nabla u|^2)^{\frac{\alpha}{2}}} \right )\ dx\\
& \qquad - \int_{\partial \Omega} a_\varepsilon(|\nabla u|) (\nabla u_i \cdot \nu) \frac{u_i \varphi^2}{(\varepsilon + |\nabla u|^2)^{\frac{\alpha}{2}}}\ d\mathcal{H}^{n-1},
\end{split}
\end{equation*}
where $\nu$ denotes the outward unit normal vector on $\partial \Omega$.\\
In addition, integrating by parts the left hand side of the previous equation, we come out with:
\begin{equation*}
\begin{split}
\int_\Omega&\frac{\hat a'_\varepsilon((\varepsilon+|\nabla u|^{2})^\frac{1}{2})}{(\varepsilon + |\nabla u|^2)^\frac{1}{2}}(\nabla u \cdot \nabla u_i) \left (\nabla u \cdot \nabla \left (\frac{u_i \varphi^2}{(\varepsilon + |\nabla u|^2)^{\frac{\alpha}{2}}}\right ) \right )\ dx\\
& \qquad + \int_{\Omega} a_\varepsilon(|\nabla u|) \nabla u_i \cdot \nabla \left ( \frac{u_i \varphi^2}{(\varepsilon + |\nabla u|^2)^{\frac{\alpha}{2}}} \right )\ dx\\
&= \int_{\partial \Omega} \frac{\hat a'_\varepsilon((\varepsilon+|\nabla u|^{2})^\frac{1}{2})}{(\varepsilon + |\nabla u|^2)^\frac{1}{2}}(\nabla u \cdot \nabla u_i) (\nabla u \cdot \nu) \frac{u_i \varphi^2}{(\varepsilon + |\nabla u|^2)^{\frac{\alpha}{2}}}\ d\mathcal{H}^{n-1}\\
& \qquad + \int_{\partial \Omega} a_\varepsilon(|\nabla u|) (\nabla u_i \cdot \nu) \frac{u_i \varphi^2}{(\varepsilon + |\nabla u|^2)^{\frac{\alpha}{2}}}\ d\mathcal{H}^{n-1}\\
&\qquad +\int_\Omega \operatorname{div}(a_\varepsilon(|\nabla u|)\nabla u)  \partial_{x_i} \left (\frac{u_i \varphi^2}{(\varepsilon + |\nabla u|^2)^{\frac{\alpha}{2}}}\right) \ dx\\
&\qquad - \int_{\partial \Omega} \operatorname{div}(a_\varepsilon(|\nabla u|)\nabla u) u_i \nu_i \frac{\varphi^2}{(\varepsilon + |\nabla u|^2)^{\frac{\alpha}{2}}},
\end{split}
\end{equation*}
where $\nu_i$ is the $i$-th component of $\nu$.\\
In order to proceed, we rewrite the previous equality as follows:
\begin{equation}\label{Equation}
\begin{split}
\int_\Omega&\frac{\hat a'_\varepsilon((\varepsilon+|\nabla u|^{2})^\frac{1}{2})}{(\varepsilon + |\nabla u|^2)^\frac{1+\alpha}{2}}(\nabla u \cdot \nabla u_i)^2 \varphi^2\ dx+ \int_{\Omega} \frac{a_\varepsilon(|\nabla u|)}{(\varepsilon + |\nabla u|^2)^{\frac{\alpha}{2}}} |\nabla u_i|^2  \varphi^2\ dx\\
&\qquad -\alpha \int_\Omega\frac{\hat a'_\varepsilon((\varepsilon+|\nabla u|^{2})^\frac{1}{2})}{(\varepsilon + |\nabla u|^2)^\frac{3+\alpha}{2}}(\nabla u \cdot \nabla u_i)(D^2u\nabla u \cdot \nabla u) u_i \varphi^2 \ dx\\
& \qquad -\alpha \int_{\Omega} \frac{a_\varepsilon(|\nabla u|)}{(\varepsilon + |\nabla u|^2)^{\frac{2+\alpha}{2}}} (D^2u \nabla u \cdot \nabla u_i) u_i \varphi^2\ dx\\
&= \int_{\partial \Omega} \frac{\hat a'_\varepsilon((\varepsilon+|\nabla u|^{2})^\frac{1}{2})}{(\varepsilon + |\nabla u|^2)^\frac{1}{2}}(\nabla u \cdot \nabla u_i) (\nabla u \cdot \nu) \frac{u_i \varphi^2}{(\varepsilon + |\nabla u|^2)^{\frac{\alpha}{2}}}\ d\mathcal{H}^{n-1}\\
& \qquad + \int_{\partial \Omega} a_\varepsilon(|\nabla u|) (\nabla u_i \cdot \nu) \frac{u_i \varphi^2}{(\varepsilon + |\nabla u|^2)^{\frac{\alpha}{2}}}\ d\mathcal{H}^{n-1}\\
&\qquad +\int_\Omega \operatorname{div}(a_\varepsilon(|\nabla u|)\nabla u)  \partial_{x_i} \left (\frac{u_i \varphi^2}{(\varepsilon + |\nabla u|^2)^{\frac{\alpha}{2}}}\right) \ dx\\
&\qquad - \int_{\partial \Omega} \operatorname{div}(a_\varepsilon(|\nabla u|)\nabla u) u_i \nu_i \frac{\varphi^2}{(\varepsilon + |\nabla u|^2)^{\frac{\alpha}{2}}} \ d\mathcal{H}^{n-1}\\
&\qquad -2 \int_\Omega \frac{\hat a'_\varepsilon((\varepsilon+|\nabla u|^{2})^\frac{1}{2})}{(\varepsilon + |\nabla u|^2)^\frac{1+\alpha}{2}}(\nabla u \cdot \nabla u_i)(\nabla u \cdot \nabla \varphi)u_i \varphi \ dx\\
&\qquad -2 \int_\Omega \frac{a_\varepsilon(|\nabla u|)}{(\varepsilon + |\nabla u|^2)^{\frac{\alpha}{2}}} (\nabla u_i \cdot \nabla \varphi) u_i  \varphi\ dx=: J_1 + \dots + J_6.\\
\end{split}
\end{equation}
Let us focus on the left-hand side of \eqref{Equation}. First of all we notice that:
\begin{equation}\label{equalities}
\sum_{i=1}^n (D^2u\nabla u \cdot \nabla u_i)u_i = \sum_{i=1}^n (\nabla u_i \cdot \nabla u)^2, \quad \sum_{i=1}^n (\nabla u \cdot \nabla u_i)u_i =(D^2u\nabla u \cdot \nabla u).
\end{equation}
Thus, summing over $i=1,...,n$ and using the previous equalities, the left hand side of \eqref{Equation} becomes:
\begin{equation}\label{comacose}
\begin{split}
\int_{\Omega}& \frac{a_\varepsilon(|\nabla u|)}{(\varepsilon + |\nabla u|^2)^{\frac{\alpha}{2}}} |D^2u|^2  \varphi^2\ dx\\
&+ \int_{\Omega} \frac{a_\varepsilon(|\nabla u|)}{(\varepsilon + |\nabla u|^2)^{\frac{2+\alpha}{2}}}\left [\frac{\hat a'_\varepsilon((\varepsilon+|\nabla u|^{2})^\frac{1}{2}) (\varepsilon + |\nabla u|^2)^\frac{1}{2}}{a_\varepsilon(|\nabla u|)} - \alpha \right] \sum_{i=1}^n (\nabla u \cdot \nabla u_i)^2 \varphi^2\ dx\\
&- \alpha \int_{\Omega} \frac{\hat a'_\varepsilon((\varepsilon+|\nabla u|^{2})^\frac{1}{2}) (\varepsilon + |\nabla u|^2)^\frac{1}{2}}{a_\varepsilon(|\nabla u|)} a_\varepsilon(|\nabla u|)(D^2u\nabla u \cdot \nabla u)^2 \frac{\varphi^2}{(\varepsilon + |\nabla u|^2)^\frac{4+\alpha}{2}} \ dx.
\end{split}
\end{equation}
For a matter of notation, we define:
\begin{equation}\label{theta_eps}
\theta_\varepsilon(x):= \frac{\hat a'_\varepsilon((\varepsilon+|\nabla u|^{2})^\frac{1}{2}) (\varepsilon + |\nabla u|^2)^\frac{1}{2}}{a_\varepsilon(|\nabla u|)},
\end{equation}
\begin{equation*}
\omega:=\frac{\nabla u}{|\nabla u|},
\end{equation*}
\begin{equation*}
H:=D^2u,
\end{equation*}
and rewrite the left hand-side of \eqref{comacose} as follows:
\begin{equation*}
\int_{\Omega} \frac{a_\varepsilon(|\nabla u|)}{(\varepsilon + |\nabla u|^2)^{\frac{\alpha}{2}}} \ \mathcal{I} \varphi^2\ dx,
\end{equation*}
where,
\begin{equation*}
\begin{split}
\mathcal{I}:&=|D^2u|^2 + ( \theta_\varepsilon(x) - \alpha) \frac{|\nabla u|^2}{(\varepsilon + |\nabla u|^2)} \sum_{i=1}^n \left(\nabla u_i \cdot \frac{\nabla u}{|\nabla u|}\right )^2\\
&\quad \qquad  - \alpha \theta_\varepsilon(x)  \frac{|\nabla u|^4}{(\varepsilon + |\nabla u|^2)^2}\left(D^2u \frac{\nabla u}{|\nabla u|}\cdot \frac{\nabla u}{|\nabla u|}\right)^2\\
&=|H|^2 +( \theta_\varepsilon(x) - \alpha) \frac{|\nabla u|^2}{(\varepsilon + |\nabla u|^2)} \ |H\omega|^2  - \alpha \theta_\varepsilon(x)  \frac{|\nabla u|^4}{(\varepsilon + |\nabla u|^2)^2}\  |H\omega \cdot \omega|^2.
\end{split}
\end{equation*}
Our aim is to prove that there exists a positive constant $C=C(\alpha,i_a)$, such that:
\begin{equation*}
\mathcal{I} \geq C |H|^2.
\end{equation*}
By \eqref{theta_eps}, we recall  $-1<i_a\leq \theta_\varepsilon (x)\leq s_a$. Let us distinguish two cases:\\
\textbf{i) } Let $x \in \Omega$ such that $\theta_\varepsilon(x)\geq 0$, then:
\begin{equation*}
\mathcal{I} \geq |H|^2+ \left[(\theta_\varepsilon(x) - \alpha) \frac{|\nabla u|^2}{(\varepsilon + |\nabla u|^2)} - \alpha \theta_\varepsilon(x)  \frac{|\nabla u|^4}{(\varepsilon + |\nabla u|^2)^2}\right] |H\omega|^2.
\end{equation*}
where we used the fact that $|H\omega \cdot \omega|^2 \leq |H\omega|^2$. For simplicity we define:
\begin{equation*}
\beta_\varepsilon(x):=(\theta_\varepsilon(x) - \alpha) \frac{|\nabla u|^2}{(\varepsilon + |\nabla u|^2)} - \alpha \theta_\varepsilon(x)  \frac{|\nabla u|^4}{(\varepsilon + |\nabla u|^2)^2}.
\end{equation*}
Therefore, if $\beta_\varepsilon(x) \geq 0$, then, obviously, $\mathcal I \geq |H|^2$; on the other hand if $\beta_\varepsilon(x) <0$, since $|H\omega|^2\leq|H|^2$,

\begin{equation*}
\begin{split}
\mathcal{I} &\geq \left[ 1+ (\theta_\varepsilon(x) - \alpha) \frac{|\nabla u|^2}{(\varepsilon + |\nabla u|^2)} - \alpha \theta_\varepsilon(x)  \frac{|\nabla u|^4}{(\varepsilon + |\nabla u|^2)^2}\right]  |H|^2\\
& = \left ( 1 +\theta_\varepsilon(x) \frac{|\nabla u|^2}{(\varepsilon + |\nabla u|^2)} \right) \left( 1 - \alpha \frac{|\nabla u|^2}{(\varepsilon + |\nabla u|^2)}\right)|H|^2\\
&\geq C |H|^2,
\end{split}
\end{equation*}
where we used that $\alpha \in [0,1)$ and $\theta_\varepsilon(x) > -1$.\\
\textbf{ii) } Let $x \in \Omega$ such that $\theta_\varepsilon(x) < 0$, by Lemma \ref{stima_punt}, it follows that:
\begin{equation*}
\begin{split}
\mathcal{I} \geq &\left( 1+ \alpha \theta_\varepsilon(x) \frac{|\nabla u|^4}{(\varepsilon + |\nabla u|^2)^2} \right)|H|^2\\
& \qquad + \left( (\theta_\varepsilon(x) - \alpha) \frac{|\nabla u|^2}{(\varepsilon + |\nabla u|^2)} -2\alpha \theta_\varepsilon(x) \frac{|\nabla u|^4}{(\varepsilon + |\nabla u|^2)^2} \right) |H\omega|^2.
\end{split}
\end{equation*}
We define:
\begin{equation*}
\delta_\varepsilon(x):=(\theta_\varepsilon(x) - \alpha) \frac{|\nabla u|^2}{(\varepsilon + |\nabla u|^2)} -2\alpha \theta_\varepsilon(x) \frac{|\nabla u|^4}{(\varepsilon + |\nabla u|^2)^2}.
\end{equation*}
Thus, if $\delta_\varepsilon(x) \geq 0$,

\begin{equation*}
\mathcal{I} \geq \left( 1+ \alpha \theta_\varepsilon(x) \frac{|\nabla u|^4}{(\varepsilon + |\nabla u|^2)^2} \right)|H|^2 \geq C |H|^2,
\end{equation*}
since, $1+ \alpha\theta_\varepsilon(x) > 0$.\\
On the other hand, if $\delta_\varepsilon(x) < 0$, since $|H\omega|^2 \leq |H|^2$, we have:
\begin{equation*}
\begin{split}
\mathcal{I}&\geq \left (1+ \alpha \theta_\varepsilon(x) \frac{|\nabla u|^4}{(\varepsilon + |\nabla u|^2)^2}+ \delta_\varepsilon(x)  \right) |H|^2\\
&=\left ( 1 +\theta_\varepsilon(x) \frac{|\nabla u|^2}{(\varepsilon + |\nabla u|^2)} \right) \left( 1 - \alpha \frac{|\nabla u|^2}{(\varepsilon + |\nabla u|^2)}\right)|H|^2\\
&\geq C|H|^2.
\end{split}
\end{equation*}
Summing up, by the previous estimates, we deduce that
\begin{equation}\label{LHS_T}
\int_{\Omega} \frac{a_\varepsilon(|\nabla u|)}{(\varepsilon + |\nabla u|^2)^{\frac{\alpha}{2}}} \ \mathcal{I} \varphi^2\ dx \geq C(\alpha,i_a) \int_\Omega \frac{a_\varepsilon(|\nabla u|)}{(\varepsilon + |\nabla u|^2)^{\frac{\alpha}{2}}} |D^2u|^2 \varphi^2 \ dx. 
\end{equation}
Regarding the right-hand side of \eqref{Equation}, we proceed in the following way:
\begin{equation*}
\begin{split}
\left|\sum_{i=1}^n J_5\right|& \leq 2C(n) \int_\Omega \left |\frac{\hat a'_\varepsilon((\varepsilon+|\nabla u|^{2})^\frac{1}{2})(\varepsilon + |\nabla u|^2)^\frac{1}{2}}{a_\varepsilon(|\nabla u|)}\right |\frac{a_\varepsilon(|\nabla u|)}{(\varepsilon + |\nabla u|^2)^\frac{\alpha}{2}}\frac{|\nabla u|^2}{(\varepsilon + |\nabla u|^2)}|D^2u||\nabla u||\varphi||\nabla \varphi| \ dx\\
&\leq 2C(n)\max\{|s_a|,|i_a|\}\int_\Omega \frac{a_\varepsilon(|\nabla u|)}{(\varepsilon + |\nabla u|^2)^\frac{\alpha}{2}}|D^2u||\nabla u||\varphi||\nabla \varphi| \ dx\\
&\leq \delta \int_\Omega \frac{a_\varepsilon(|\nabla u|)}{(\varepsilon + |\nabla u|^2)^\frac{\alpha}{2}}|D^2u|^2 \varphi^2 \ dx + \frac{\max\{|s_a|,|i_a|\}^2C(n)^2}{\delta}\int_\Omega \frac{a_\varepsilon(|\nabla u|)}{(\varepsilon + |\nabla u|^2)^\frac{\alpha}{2}}|\nabla u|^2 |\nabla \varphi|^2 \ dx,
\end{split}
\end{equation*}
where, in the last inequality, we used the weighted Young inequality, and $C(n)$ is a positive constant.\\
We can estimate, as before, the term $J_6$:
\begin{equation*}
\begin{split}
\left|\sum_{i=1}^n J_6\right|& \leq 2C(n) \int_\Omega \frac{a_\varepsilon(|\nabla u|)}{(\varepsilon + |\nabla u|^2)^\frac{\alpha}{2}}|D^2u||\nabla u||\varphi||\nabla \varphi| \ dx\\
&\leq \delta \int_\Omega \frac{a_\varepsilon(|\nabla u|)}{(\varepsilon + |\nabla u|^2)^\frac{\alpha}{2}}|D^2u|^2 \varphi^2 \ dx + \frac{C(n)^2}{\delta}\int_\Omega \frac{a_\varepsilon(|\nabla u|)}{(\varepsilon + |\nabla u|^2)^\frac{\alpha}{2}}|\nabla u|^2 |\nabla \varphi|^2 \ dx.
\end{split}
\end{equation*}
Therefore, we get:
\begin{equation}\label{RHS_T}
\begin{split}
\left|\sum_{i=1}^n J_5\right| + \left|\sum_{i=1}^n J_6\right| \leq 2 \delta &\int_\Omega \frac{a_\varepsilon(|\nabla u|)}{(\varepsilon + |\nabla u|^2)^\frac{\alpha}{2}}|D^2u|^2 \varphi^2 \ dx\\
& + \frac{C(n)^2(\max\{|s_a|,|i_a|\}^2+1)}{\delta}\int_\Omega \frac{a_\varepsilon(|\nabla u|)}{(\varepsilon + |\nabla u|^2)^\frac{\alpha}{2}}|\nabla u|^2 |\nabla \varphi|^2 \ dx.
\end{split}
\end{equation}
The remaining part consists in estimating the boundary terms appearing in \eqref{Equation}, $J_1$, $J_2$ and $J_4$. For a matter of simplicity, let us treat separately the Dirichlet, the Neumann and the convex cases.

\vspace{0.3cm}

\textbf{Dirichlet case.} Since $u=0$ on $\partial \Omega$, then we can write explicitly the outward unit normal vector on the boundary as:
\begin{equation}\label{Nu_Dir}
\nu = \frac{\nabla u}{|\nabla u|}, \quad \text{if} \ \nabla u \neq 0,
\end{equation}
otherwise if $\nabla u = 0$, the corresponding boundary term will be identically equal to zero.\\
First of all we rewrite the term $J_4$ in a more useful way, namely:
\begin{equation}\label{J_4}
\begin{split}
 \sum_{i=1}^n J_4 &=- \int_{\partial \Omega} \operatorname{div}(a_\varepsilon(|\nabla u|)\nabla u) \sum_{i=1}^n u_i \nu_i \frac{\varphi^2}{(\varepsilon + |\nabla u|^2)^{\frac{\alpha}{2}}}\ d\mathcal{H}^{n-1}\\
&=-\int_{\partial \Omega} \frac{a_\varepsilon(|\nabla u|)}{(\varepsilon - |\nabla u|^2)^{\frac{\alpha}{2}}} \Delta u |\nabla u| \varphi^2\ d\mathcal{H}^{n-1}\\
&\quad  \qquad- \int_{\partial \Omega} \frac{\hat a'_\varepsilon((\varepsilon+|\nabla u|^{2})^\frac{1}{2})}{(\varepsilon + |\nabla u|^2)^\frac{1+\alpha}{2}}(D^2u\nabla u \cdot \nabla u)|\nabla u| \varphi^2 \ d\mathcal{H}^{n-1}.
\end{split}
\end{equation}
By \eqref{equalities}, we deduce:
\begin{equation}\label{J_1}
\begin{split}
\sum_{i=1}^n J_1 &= \int_{\partial \Omega} \frac{\hat a'_\varepsilon((\varepsilon+|\nabla u|^{2})^\frac{1}{2})}{(\varepsilon + |\nabla u|^2)^\frac{1+\alpha}{2}}(D^2u\nabla u \cdot \nabla u) (\nabla u \cdot \nu) \varphi^2 \ d\mathcal{H}^{n-1}\\
&=\int_{\partial \Omega} \frac{\hat a'_\varepsilon((\varepsilon+|\nabla u|^{2})^\frac{1}{2})}{(\varepsilon + |\nabla u|^2)^\frac{1+\alpha}{2}}(D^2u\nabla u \cdot \nabla u)|\nabla u| \varphi^2 \ d\mathcal{H}^{n-1}.
\end{split}
\end{equation}
To evaluate the boundary term $J_2$, by \cite[Equation (3,1,1,8)]{Gr}, we recall that:
\begin{equation}\label{Nonsmooth_dom}
\begin{split}
    \Delta u \frac{\partial u}{\partial {\nu}} - &\sum_{i,j=1}^n u_{ij} u_i \nu_j\\
    &= \operatorname{div}_T \left ( \frac{\partial u}{\partial  \nu} \nabla_T u\right ) - \operatorname{tr}\mathcal B \left (\frac{\partial u}{\partial  \nu} \right)^2 \\
    &\qquad- \mathcal{B}(\nabla_Tu,\nabla_Tu) - 2 \nabla_Tu \cdot \nabla_T\frac{\partial u}{\partial  \nu} \quad \text{on } \partial \Omega,
\end{split}
\end{equation}
where $\mathcal{B}$ stands for the second fundamental form associated to the boundary, $\operatorname{tr}\mathcal B$ for its trace, $\operatorname{div}_T$ and $\nabla_T$ denote the divergence and the gradient operator on $\partial\Omega$.\\
The Dirichlet homogeneous boundary condition, implies that $\nabla_T u =0$ on $\partial \Omega$, thus, by \eqref{Nonsmooth_dom}, $J_2$ becomes:

\begin{equation}\label{J_2}
\begin{split}
\sum_{i=1}^n J_2 &= \int_{\partial \Omega} \frac{a_\varepsilon(|\nabla u|) }{(\varepsilon + |\nabla u|^2)^{\frac{\alpha}{2}}}\sum_{i=1}^n \sum_{j=1}^n u_{ij} u_i \nu_j \varphi^2\ d\mathcal{H}^{n-1}\\
&= \int_{\partial \Omega} \frac{a_\varepsilon(|\nabla u|) }{(\varepsilon + |\nabla u|^2)^{\frac{\alpha}{2}}}\Delta u |\nabla u| \varphi^2\ d\mathcal{H}^{n-1}\\
&\quad  \qquad + \int_{\partial \Omega} \frac{a_\varepsilon(|\nabla u|) }{(\varepsilon + |\nabla u|^2)^{\frac{\alpha}{2}}}\operatorname{tr} \mathcal{B} |\nabla u|^2 \varphi^2\ d\mathcal{H}^{n-1}.
\end{split}
\end{equation} 
Finally, by \eqref{J_1}, \eqref{J_2} and \eqref{J_4}, we get the resulting boundary term:
\begin{equation*}
\begin{split}
\sum_{i=1}^n (J_1 + J_2 + J_4) &= \int_{\partial \Omega} \frac{a_\varepsilon(|\nabla u|) }{(\varepsilon + |\nabla u|^2)^{\frac{\alpha}{2}}}\operatorname{tr} \mathcal{B} |\nabla u|^2 \varphi^2\ d\mathcal{H}^{n-1}\\
&\leq C(n) \int_{\partial \Omega} \frac{a_\varepsilon(|\nabla u|) }{(\varepsilon + |\nabla u|^2)^{\frac{\alpha}{2}}}|\mathcal{B}| |\nabla u|^2 \varphi^2\ d\mathcal{H}^{n-1},
\end{split}
\end{equation*}
where $C(n)$ is a positive constant. Now we assume that $\varphi \in C^{\infty}_c(B_r(x_0))$ non negative, for some $x_0\in \partial \Omega$ and $r\in (0,R_{\Omega})$. By \eqref{antony}, with $\rho =|\mathcal{B}|$ and $v=\varphi\frac{a_\varepsilon(|\nabla u|)^\frac{1}{2}}{(\varepsilon+|\nabla u|^2)^{\frac{\alpha}{4}}} u_i \in W_0^{1,2}(B_r(x_0))$, since $$|\nabla v|\leq C(\alpha,i_a,s_a)\varphi\frac{a_\varepsilon(|\nabla u|)^\frac{1}{2}}{(\varepsilon+|\nabla u|^2)^{\frac{\alpha}{4}}} |\nabla u_i|+ \frac{a_\varepsilon(|\nabla u|)^\frac{1}{2}}{(\varepsilon+|\nabla u|^2)^{\frac{\alpha}{4}}} |u_i||\nabla \varphi|,$$ for $i=1,...,n$, with $C(\alpha,i_a,s_a)$ positive constant, we deduce that:
\begin{equation}\label{boundary_T}
\begin{split}
\int_{\partial \Omega \cap B_r(x_0)} &\frac{a_\varepsilon(|\nabla u|) }{(\varepsilon + |\nabla u|^2)^{\frac{\alpha}{2}}}|\mathcal{B}| |\nabla u|^2 \varphi^2\ d\mathcal{H}^{n-1}\\
&\leq C(n,s_a,i_a,\alpha,L_\Omega)K_\Omega(r) \left ( \int_{\Omega \cap B_r (x_0)} \frac{a_\varepsilon(|\nabla u|) }{(\varepsilon + |\nabla u|^2)^{\frac{\alpha}{2}}} |D^2u|^2 \varphi^2\ dx\right.\\
&\left. \qquad \qquad \qquad+ \int_{\Omega \cap B_r (x_0)} \frac{a_\varepsilon(|\nabla u|) }{(\varepsilon + |\nabla u|^2)^{\frac{\alpha}{2}}} |\nabla u|^2 |\nabla \varphi|^2\ dx \right),
\end{split}
\end{equation}
for some positive constant $C(n,s_a,i_a,\alpha,L_\Omega)$.\\
By \eqref{Equation}, \eqref{LHS_T}, \eqref{RHS_T} and \eqref{boundary_T}, we obtain:

\begin{equation*}
\begin{split}
&\left (C(\alpha,i_a) - 2\delta - C(n,s_a,i_a,\alpha,L_\Omega)K_\Omega(r) \right)\int_\Omega \frac{a_\varepsilon(|\nabla u|)}{(\varepsilon + |\nabla u|^2)^{\frac{\alpha}{2}}} |D^2u|^2 \varphi^2 \ dx\\
&\qquad \leq \left ( \frac{C(n,s_a,i_a)}{\delta} + C(n,s_a,i_a,\alpha,L_\Omega)K_\Omega(r) \right ) \int_{\Omega} \frac{a_\varepsilon(|\nabla u|) }{(\varepsilon + |\nabla u|^2)^{\frac{\alpha}{2}}} |\nabla u|^2 |\nabla \varphi|^2 \ dx\\
&\qquad \qquad + \sum_{i=1}^n \int_\Omega \operatorname{div}(a_\varepsilon(|\nabla u|)\nabla u)  \partial_{x_i} \left (\frac{u_i \varphi^2}{(\varepsilon + |\nabla u|^2)^{\frac{\alpha}{2}}}\right) \ dx\\
&\qquad = \left ( \frac{C(n,s_a,i_a)}{\delta} + C(n,s_a,i_a,\alpha,L_\Omega)K_\Omega(r) \right ) \int_{\Omega} \frac{a_\varepsilon(|\nabla u|) }{(\varepsilon + |\nabla u|^2)^{\frac{\alpha}{2}}} |\nabla u|^2 |\nabla \varphi|^2 \ dx\\
&\qquad \qquad + \sum_{i=1}^n \int_\Omega \operatorname{div}(a_\varepsilon(|\nabla u|)\nabla u)  \partial_{x_i} \left (\frac{u_i}{(\varepsilon + |\nabla u|^2)^{\frac{\alpha}{2}}}\right)\varphi^2 \ dx\\
&\qquad \qquad + 2\int_\Omega \operatorname{div}(a_\varepsilon(|\nabla u|)\nabla u) \frac{(\nabla u \cdot \nabla \varphi)}{(\varepsilon + |\nabla u|^2)^{\frac{\alpha}{2}}}\varphi \ dx\\
&\qquad \leq \left ( \frac{C(n,s_a,i_a)}{\delta} + C(n,s_a,i_a,\alpha,L_\Omega)K_\Omega(r) \right ) \int_{\Omega} \frac{a_\varepsilon(|\nabla u|) }{(\varepsilon + |\nabla u|^2)^{\frac{\alpha}{2}}} |\nabla u|^2 |\nabla \varphi|^2 \ dx\\
&\qquad \qquad + \sum_{i=1}^n \int_\Omega \operatorname{div}(a_\varepsilon(|\nabla u|)\nabla u)  \partial_{x_i} \left (\frac{u_i}{(\varepsilon + |\nabla u|^2)^{\frac{\alpha}{2}}}\right)\varphi^2 \ dx\\
&\qquad \qquad + \int_\Omega \operatorname{div}(a_\varepsilon(|\nabla u|)\nabla u) \frac{|\nabla u|}{(\varepsilon + |\nabla u|^2)^{\frac{\alpha}{2}}}\varphi^2 \ dx\\
&\qquad \qquad + \int_\Omega \operatorname{div}(a_\varepsilon(|\nabla u|)\nabla u) \frac{|\nabla u|}{(\varepsilon + |\nabla u|^2)^{\frac{\alpha}{2}}}|\nabla \varphi|^2 \ dx,
\end{split}
\end{equation*}
where in the last inequality we have used a standard Young's inequality.
Moreover, if condition \eqref{condizionesuk} is satisfied with $\overline c=C(\alpha,i_a)/C(n,s_a,i_a,\alpha,L_\Omega)$, then, for $\delta>0$ sufficiently small, there exists a constant $C>0$, depending on $\Omega$ only through $L_\Omega$ and $r'\in (0,R_\Omega)$ depending also on $K$, such that:
\begin{equation*}
\begin{split}
    C(\alpha,i_a)& - 2\delta - C(n,s_a,i_a,\alpha,L_\Omega)K_\Omega(r) \\
    &\geq C(\alpha,i_a) - 2\delta - C(n,s_a,i_a,\alpha,L_\Omega)K(r)\geq C,
\end{split}
\end{equation*}
with $r \in (0,r']$.\\
Consequently, we get the following estimate:
\begin{equation}\label{MAIN_e}
\begin{split}
&\int_\Omega \frac{a_\varepsilon(|\nabla u|)}{(\varepsilon + |\nabla u|^2)^{\frac{\alpha}{2}}} |D^2u|^2 \varphi^2 \ dx\\
&\qquad \leq \mathcal{C} \left (\int_{\Omega} \frac{a_\varepsilon(|\nabla u|) }{(\varepsilon + |\nabla u|^2)^{\frac{\alpha}{2}}} |\nabla u|^2 |\nabla \varphi|^2 \ dx \right.\\
&\left.\qquad \qquad + \sum_{i=1}^n \int_\Omega \operatorname{div}(a_\varepsilon(|\nabla u|)\nabla u)  \partial_{x_i} \left (\frac{u_i}{(\varepsilon + |\nabla u|^2)^{\frac{\alpha}{2}}}\right)\varphi^2 \ dx\right.\\
&\left.\qquad \qquad + \int_\Omega \operatorname{div}(a_\varepsilon(|\nabla u|)\nabla u) \frac{|\nabla u|}{(\varepsilon + |\nabla u|^2)^{\frac{\alpha}{2}}}\varphi^2 \ dx \right.\\
&\left. \qquad \qquad + \int_\Omega \operatorname{div}(a_\varepsilon(|\nabla u|)\nabla u) \frac{|\nabla u|}{(\varepsilon + |\nabla u|^2)^{\frac{\alpha}{2}}}|\nabla \varphi|^2 \ dx\right),
\end{split}
\end{equation}
for some positive constant $\mathcal{C}(n,s_a,i_a,\alpha,L_\Omega)$.\\
In the case when $x_0\in \Omega$ and $B_r(x_0)\subset\Omega$ inequality \eqref{MAIN_e} holds. Indeed, the boundary term
in \eqref{boundary_T} vanishes. Moreover the constant $\mathcal{C}$ is independent of $L_\Omega$.

Otherwise, there exists $r'' \in (0,r')$, hence depending on $L_\Omega$, $d_\Omega$ (the diameter of $\Omega$), and $K$ such that $\overline \Omega$ admits a finite covering $\{B_{r_k}\}$, with $r''\leq r_k \leq r'$, and a family of functions $\{\varphi_{r_k}\}$
such that $\varphi_{r_k} \in C_c^{\infty}(B_{r_k})$ and $\{\varphi^2_{r_k}\}$ is a partition of unity of $\overline \Omega$ associated with the covering $\{B_{r_k}\}$. Thus, $\sum_k \varphi^2_{r_k}=1$ in $\overline \Omega$. Moreover the functions $\varphi_{r_k}$ can also chosen so that $|\nabla \varphi_{r_k}|\leq C/r_k\leq C / r''$. Using \eqref{MAIN_e} with $\varphi=\varphi_{r_k}$, we obtain:
\begin{equation}\label{FINE_D}
\begin{split}
&\int_\Omega \frac{a_\varepsilon(|\nabla u|)}{(\varepsilon + |\nabla u|^2)^{\frac{\alpha}{2}}} |D^2u|^2 \ dx\\
&\qquad \leq \mathcal{C} \left (\int_{\Omega} \frac{a_\varepsilon(|\nabla u|) }{(\varepsilon + |\nabla u|^2)^{\frac{\alpha}{2}}} |\nabla u|^2\ dx \right.\\
&\left.\qquad \qquad + \sum_{i=1}^n \int_\Omega \operatorname{div}(a_\varepsilon(|\nabla u|)\nabla u)  \partial_{x_i} \left (\frac{u_i}{(\varepsilon + |\nabla u|^2)^{\frac{\alpha}{2}}}\right)\ dx\right.\\
&\left. \qquad \qquad + \int_\Omega \operatorname{div}(a_\varepsilon(|\nabla u|)\nabla u) \frac{|\nabla u|}{(\varepsilon + |\nabla u|^2)^{\frac{\alpha}{2}}} \ dx\right),
\end{split}
\end{equation}
where $\mathcal{C}=\mathcal{C}(n,s_a,i_a,\alpha,L_\Omega,d_\Omega,K)$ is a positive constant.

\vspace{0.3cm}

\textbf{Neumann case.} The approach remains the same for the case with homogeneous Neumann boundary conditions. Below, we highlight the key modifications in the proof. First of all we notice that, since $\frac{\partial u}{\partial \nu}=0$, then the boundary terms, $J_1$ and $J_4$, vanish identically. We are left to deal with the following term:
\begin{equation*}
\begin{split}
\sum_{i=1}^n J_2 &= \int_{\partial \Omega} \frac{a_\varepsilon(|\nabla u|) }{(\varepsilon + |\nabla u|^2)^{\frac{\alpha}{2}}}\sum_{i=1}^n \sum_{j=1}^n u_{ij} u_i \nu_j \varphi^2\ d\mathcal{H}^{n-1}\\
&= \int_{\partial \Omega} \frac{a_\varepsilon(|\nabla u|) }{(\varepsilon + |\nabla u|^2)^{\frac{\alpha}{2}}}\mathcal{B}(\nabla_Tu,\nabla_Tu) \varphi^2\ d\mathcal{H}^{n-1}\\
&\leq C(n)  \int_{\partial \Omega} \frac{a_\varepsilon(|\nabla u|) }{(\varepsilon + |\nabla u|^2)^{\frac{\alpha}{2}}}|\mathcal{B}||\nabla u|^2\varphi^2\ d\mathcal{H}^{n-1},
\end{split}
\end{equation*}
where we used equality \eqref{Nonsmooth_dom}.\\
These variations allow the procedure used in the Dirichlet case to yield:
\begin{equation}\label{Neu_quasi}
\begin{split}
&\int_\Omega \frac{a_\varepsilon(|\nabla u|)}{(\varepsilon + |\nabla u|^2)^{\frac{\alpha}{2}}} |D^2u|^2\varphi^2 \ dx\\
&\qquad \leq \mathcal{C} \left (\int_{\Omega} \frac{a_\varepsilon(|\nabla u|) }{(\varepsilon + |\nabla u|^2)^{\frac{\alpha}{2}}} |\nabla u|^2|\nabla \varphi|^2\ dx \right.\\
&\left.\qquad \qquad + \sum_{i=1}^n \int_\Omega \operatorname{div}(a_\varepsilon(|\nabla u|)\nabla u)  \partial_{x_i} \left (\frac{u_i \varphi^2}{(\varepsilon + |\nabla u|^2)^{\frac{\alpha}{2}}}\right)\ dx\right),
\end{split}
\end{equation}
where $\mathcal{C}=\mathcal{C}(n,s_a,i_a,\alpha,L_\Omega,d_\Omega,K)$ is a positive constant.\\
Integrating by parts with respect to the variable $x_i$ the last term of \eqref{Neu_quasi} and using the Neumann homogeneous boundary condition, we obtain:
\begin{equation}\label{FINE_N}
\begin{split}
&\int_\Omega \frac{a_\varepsilon(|\nabla u|)}{(\varepsilon + |\nabla u|^2)^{\frac{\alpha}{2}}} |D^2u|^2 \varphi^2 \ dx\\
&\qquad \leq \mathcal{C} \left (\int_{\Omega} \frac{a_\varepsilon(|\nabla u|) }{(\varepsilon + |\nabla u|^2)^{\frac{\alpha}{2}}} |\nabla u|^2 |\nabla \varphi|^2\ dx \right.\\
&\left.\qquad \qquad - \sum_{i=1}^n \int_\Omega \partial_{x_i} \operatorname{div}(a_\varepsilon(|\nabla u|)\nabla u) \left (\frac{u_i}{(\varepsilon + |\nabla u|^2)^{\frac{\alpha}{2}}}\right) \varphi^2\ dx\right).
\end{split}
\end{equation}
In conclusion, the partition of unity and a further integration by parts, with respect to the variable $x_i$, yield the desired result.

\vspace{0.3cm}

\textbf{Convex case.} If $\Omega$ is convex, then, due to the sign of the boundary curvatures, namely:
\begin{equation*}
\operatorname{tr}\mathcal{B} \leq 0, \quad \mathcal{B}(\nabla_{T} u,\nabla_{T} u) \leq 0 \quad \text{on } \partial{\Omega},
\end{equation*}
we can disregard the boundary integrals appearing both in the Dirichlet and in the Neumann estimates. Therefore, via the same procedure, both inequalities \eqref{FINE_D} and \eqref{FINE_N} hold, with a positive constant independent on $K_\Omega$.
\end{proof}
\vspace{0.3cm}
To proceed, we recall the following technical lemma (see \cite[Lemma 5.2]{Cma} or \cite{Ant}):
\vspace{0.1cm}

\begin{lemma}\label{approxcap}
Let $\Omega$ be a bounded Lipschitz domain in $\R^n$, $n \geq 2$ such that $\partial \Omega \in W^{2,1}$. Assume that the function $\mathcal K_\Omega (r)$, defined as in \eqref{KB}, is finite-valued for $r\in (0,1)$.
Then there exist positive constants $r_0$ and $C$ and a sequence of bounded open sets $\{\Omega_m\}$, 
such that $\partial \Omega _m \in C^\infty$, $\Omega \subset \Omega _m$, $\lim _{m \to \infty}|\Omega _m \setminus \Omega| = 0$, the Hausdorff distance between $\Omega _m$ and $\Omega$ tends to $0$ as $m \to \infty$,
\begin{equation}\label{prop1}
L_{\Omega _m} \leq C L_\Omega \,, \quad d_{\Omega _m} \leq C d_\Omega
\end{equation}
and
\begin{equation}\label{prop2}
\mathcal K_{\Omega_m}(r) \leq C \mathcal K_{\Omega} (r)
\end{equation}
for $r\in (0, r_0)$ and $m \in \mathbb N$.\\

\end{lemma}
\section{Proofs of the main results}\label{second-order}
Notice that our focus is on the case when the gradient is sufficiently small, as the behavior in regions where it is bounded away from zero is already well understood, being covered by the classical regularity theory. Therefore, in the following, we will use \eqref{estimate_a_eps} with $t<<1$.

\vspace{0.1 cm}

\begin{proof}[Proof of Theorem \ref{teo1INTRO}.]
The thesis formally follows from an application of Theorem \ref{STIMA}. However, in our setting, the lack of regularity of the weak solution $u$ and of the domain $\Omega$ prevents a direct application of previous results, requiring instead a regularization argument involving both the domain and the right-hand side of the equation.

\textbf{Dirichlet case.} First of all, we analyze \eqref{Dir_pb}, which solution obeys the Dirichlet homogeneous boundary condition. We split the proof into several steps.

\textit{Step $1.$} Here, we assume 
\begin{equation}\label{bordoregolare}
   \partial \Omega\in C^\infty,
\end{equation}
\vspace{0.05cm}
\begin{equation}\label{reg_f}
    f \in C^\infty(\Omega).
\end{equation}
Given $\varepsilon$ small enough, we consider a weak solution ${u_\varepsilon}$ to the approximating problem
\beq\label{system33}\tag{$p$-$D_\varepsilon$}
\begin{cases}
-\operatorname{ div}( a_\varepsilon(|\nabla u_\varepsilon|)\nabla{ u_\varepsilon})=   { f}(x)& \mbox{in $\Omega$}\\
 u_\varepsilon = 0  & \mbox{on  $\partial \Omega$},
\end{cases}
\eeq
where $a_\varepsilon$ is defined as in \eqref{a_cap}.\\
By \cite[Theorem $1.3$]{Cma_S} (see also \cite{Cma2}), we remark that:
\begin{equation}\label{linfinito}
    \|\nabla{u}_\varepsilon\|_{L^{\infty}(\Omega)}\leq C(n,i_a,s_a,|\Omega|,\|\operatorname{tr}\mathcal{B}\|_{L^{n-1,1}(\Omega)},\| f\|_{L(\Omega)}),
\end{equation}
where $C(n,i_a,s_a,|\Omega|,\|\operatorname{tr}\mathcal{B}\|_{L^{n-1,1}(\Omega)},\| f\|_{L(\Omega)})$ is a positive constant and $L(\Omega)$ is defined in \eqref{L_Omega}.\\
Furthermore, by \cite[Theorem $1.7$]{Li2}, there exist constants $\alpha \in [0,1)$ and a positive constant $C=C(n,s_a,i_a,\Omega,\|f\|_{L^\infty(\Omega)})$ such that:
\begin{equation*}
\|u_\varepsilon\|_{C^{1,\alpha}(\overline \Omega)} \leq C.
\end{equation*}
Moreover, by classical regularity, see \cite[Section $8.2$]{GIU}, weak solutions $u_\varepsilon$ of \eqref{system33} belong to $W^{2,2}(\Omega)$. This allows us to call for an iterative procedure, see \cite[Chapter $4$, Section $6$]{LU_S}, that leads to $u_\varepsilon \in C^3(\overline \Omega)$.

The regularity of $u_\varepsilon$ and the domain $\Omega$ enables us to apply the estimate \eqref{Dir_estimate} of Theorem \ref{STIMA}:
\begin{equation}\label{Dir_u_eps}
\begin{split}
&\int_\Omega \frac{a_\varepsilon(|\nabla u_\varepsilon|)}{(\varepsilon + |\nabla u_\varepsilon|^2)^{\frac{\alpha}{2}}} |D^2u_\varepsilon|^2 \ dx\\
&\qquad \leq \mathcal{C} \left (\int_{\Omega} \frac{a_\varepsilon(|\nabla u_\varepsilon|) }{(\varepsilon + |\nabla u_\varepsilon|^2)^{\frac{\alpha}{2}}} |\nabla u_\varepsilon|^2\ dx \right.\\
&\left.\qquad \qquad + \sum_{i=1}^n \int_\Omega \operatorname{div}(a_\varepsilon(|\nabla u_\varepsilon|)\nabla u_\varepsilon)  \partial_{x_i} \left (\frac{u_{\varepsilon,i}}{(\varepsilon + |\nabla u_\varepsilon|^2)^{\frac{\alpha}{2}}}\right)\ dx\right.\\
&\left. \qquad \qquad + \int_\Omega \operatorname{div}(a_\varepsilon(|\nabla u_\varepsilon|)\nabla u_\varepsilon) \frac{|\nabla u_\varepsilon|}{(\varepsilon + |\nabla u_\varepsilon|^2)^{\frac{\alpha}{2}}} \ dx\right),
\end{split}
\end{equation}
where $\mathcal{C}(n,i_a,s_a,\alpha,L_\Omega,d_\Omega,K_{\Omega})$ is a positive constant.\\
Recalling that $-\operatorname{ div}(a_\varepsilon(|\nabla u_\varepsilon|)\nabla{ u_\varepsilon})=   { f}(x)$, we deduce:
\begin{equation}\label{aia}
    \begin{split}
&\int_\Omega \frac{a_\varepsilon(|\nabla u_\varepsilon|)}{(\varepsilon + |\nabla u_\varepsilon|^2)^{\frac{\alpha}{2}}} |D^2u_\varepsilon|^2 \ dx\\
&\qquad \leq \mathcal{C} \left (\int_{\Omega} \frac{a_\varepsilon(|\nabla u_\varepsilon|) }{(\varepsilon + |\nabla u_\varepsilon|^2)^{\frac{\alpha}{2}}} |\nabla u_\varepsilon|^2\ dx \right.\\
&\left.\qquad \qquad - \sum_{i=1}^n \int_\Omega f \cdot \partial_{x_i} \left (\frac{u_{\varepsilon,i}}{(\varepsilon + |\nabla u_\varepsilon|^2)^{\frac{\alpha}{2}}}\right)\ dx\right.\\
&\left. \qquad \qquad - \int_\Omega f \frac{|\nabla u_\varepsilon|}{(\varepsilon + |\nabla u_\varepsilon|^2)^{\frac{\alpha}{2}}} \ dx\right),
\end{split}
\end{equation}
To proceed, we treat the second term in the right hand side of \eqref{aia} separately:
\begin{equation*}
    F:=-\sum_{i=1}^n \int_{\Omega} f \cdot \partial_{x_i}\left(\frac{{u}_{\varepsilon,i}}{(\varepsilon+|\nabla{ u}_\varepsilon|^2)^{\frac{\alpha}{2}}}\right) \,dx.
\end{equation*}
Integrating by parts, we get:
\begin{equation*}
\begin{split}
   -\sum_{i=1}^n \int_{\Omega} f \cdot \partial_{x_i}\left(\frac{{u}_{\varepsilon,i}}{(\varepsilon+|\nabla{ u}_\varepsilon|^2)^{\frac{\alpha}{2}}}\right) \,dx=
    &\sum_{i=1}^n \int_{\Omega} \partial_{x_i} f \cdot\left(\frac{{u}_{\varepsilon,i}}{(\varepsilon+|\nabla{ u}_\varepsilon|^2)^{\frac{\alpha}{2}}}\right) \,dx\\
    &- \sum_{i=1}^n \int_{\partial \Omega} f \cdot \left(\frac{{u}_{\varepsilon,i}}{(\varepsilon+|\nabla{ u}_\varepsilon|^2)^{\frac{\alpha}{2}}}\right) \nu_i \,d\mathcal{H}^{n-1}.
\end{split}
\end{equation*}
Notice that, since $u_\varepsilon \in C^1(\overline \Omega)$, $\|\nabla u_\varepsilon\|_{C^{0}(\overline \Omega)}=\|\nabla u_\varepsilon\|_{L^\infty(\Omega)}$. Therefore, by \eqref{linfinito}, we have the following estimate:
\begin{equation}\label{bound_G}
\begin{split}
    |F| &\leq \sum_{i=1}^n \int_{\Omega} |\partial_{x_i} f||\nabla u_\varepsilon|^{1-\alpha}\,dx +\|\nabla u_\varepsilon\|_{L^\infty(\Omega)}^{1-\alpha} \int_{\partial\Omega} |{f}| \,d\mathcal{H}^{n-1}\\
    &\leq \tilde{C} \left (\int_{\Omega} |\nabla{f}|\, dx+ \int_{\partial\Omega} |{f}| \,d\mathcal{H}^{n-1} \right),
\end{split}
\end{equation}
with $\tilde{C}=\tilde{C}(n,i_a,s_a,|\Omega|,\|\operatorname{tr}\mathcal{B}\|_{L^{n-1,1}(\Omega)},\| f\|_{L(\Omega)},\alpha)$.

By \eqref{bound_G} and a trace inequality in Lipschitz domains, see \cite{Nec}, we can rewrite \eqref{aia} as follows:
\begin{equation}\label{aia2}
    \begin{split}
&\int_\Omega \frac{a_\varepsilon(|\nabla u_\varepsilon|)}{(\varepsilon + |\nabla u_\varepsilon|^2)^{\frac{\alpha}{2}}} |D^2u_\varepsilon|^2 \ dx\\
&\qquad \leq \mathcal{C} \left (\int_{\Omega} \frac{a_\varepsilon(|\nabla u_\varepsilon|) }{(\varepsilon + |\nabla u_\varepsilon|^2)^{\frac{\alpha}{2}}} |\nabla u_\varepsilon|^2\ dx  + \|f\|_{W^{1,1}(\Omega)} \right).
\end{split}
\end{equation}
where $\mathcal{C}=\mathcal{C}(n,i_a,s_a,\alpha,L_\Omega,d_\Omega,K_{\Omega},\| f\|_{L(\Omega)},|\Omega|,\|\operatorname{tr}\mathcal{B}\|_{L^{n-1,1}(\Omega)})$ is a positive constant.

Finally, again by \eqref{linfinito} and \eqref{estimate_a_eps}, recalling that $\alpha\in [0,1),$ we deduce:
\begin{equation}\label{second_deriv_eps}
    \int_\Omega \frac{a_\varepsilon(|\nabla u_\varepsilon|)}{(\varepsilon + |\nabla u_\varepsilon|^2)^{\frac{\alpha}{2}}} |D^2u_\varepsilon|^2 \ dx\leq \mathcal{C},
\end{equation}
where $\mathcal{C}=\mathcal{C}(n,i_a,s_a,\alpha,L_\Omega,d_\Omega,K_{\Omega},|\Omega|,\|\operatorname{tr}\mathcal{B}\|_{L^{n-1,1}(\Omega)},\| f\|_{L(\Omega)},\| f\|_{W^{1,1}(\Omega)})$.\\
Now, under assumptions on $k$ (see \eqref{T_k}), we claim
\begin{equation}\label{bound_quasi_stress_field1}
\left \|a_\varepsilon(|\nabla u_\varepsilon|)^{k}\nabla u_{\varepsilon} \right \|_{W^{1,2}(\Omega)}\leq C,
\end{equation}
where $C=C(n,i_a,s_a,\alpha,L_\Omega,d_\Omega,\Psi_{\Omega},|\Omega|,\|\operatorname{tr}\mathcal{B}\|_{L^{n-1,1}(\Omega)},\| f\|_{L(\Omega)},\| f\|_{W^{1,1}(\Omega)})$ is a positive constant.
To move forward, let us distinguish different cases:\\
\textbf{Case 1.} If $0 \leq i_a < s_a$, by \eqref{estimate_a_eps}, we have:
\begin{equation*}
    \frac{1}{(\varepsilon + |\nabla u_\varepsilon|^2)^{\frac{\alpha}{2}}} \geq \frac{a(1)^{\frac{\alpha}{s_a}}}{a_\varepsilon(|\nabla u_\varepsilon|)^{\frac{\alpha}{s_a}}}.
\end{equation*}
Therefore, by \eqref{second_deriv_eps}, we obtain the following estimate
\begin{equation}\label{before_gamma}
    \int_\Omega a_\varepsilon(|\nabla u_\varepsilon|)^{1-\frac{\alpha}{s_a}} |D^2u_\varepsilon|^2 \ dx\leq \mathcal{C}.
\end{equation}
Let $\gamma \geq 0$. By \eqref{estimate_a_eps} and \eqref{linfinito} we note that
\begin{equation}\label{Est_n}
    \frac{a(1)^\gamma}{a_\varepsilon(|\nabla u_\varepsilon|)^\gamma} \geq \frac{1}{(\varepsilon + |\nabla u_\varepsilon|^2)^{\frac{i_a}{2}\gamma}} \geq C(\|\nabla u_\varepsilon\|_{L^\infty(\Omega)},i_a,\gamma),
\end{equation} 
where $C(\|\nabla u_\varepsilon\|_{L^\infty(\Omega)},i_a,\gamma)$ is a positive constant.\\
By \eqref{before_gamma} and \eqref{Est_n}, for every $\gamma\geq 0$, we have
\begin{equation}\label{after_gamma}
\begin{split}
C(\|\nabla u_\varepsilon\|_{L^\infty(\Omega)},i_a,\gamma)&\int_\Omega {a_\varepsilon(|\nabla u_\varepsilon|)^{1-\frac{\alpha}{s_a}+ \gamma}} |D^2u_\varepsilon|^2 \ dx \\&\leq\int_\Omega \frac{a_\varepsilon(|\nabla u_\varepsilon|)^{1-\frac{\alpha}{s_a}+ \gamma}}{a_\varepsilon(|\nabla u_\varepsilon|)^\gamma} |D^2u_\varepsilon|^2 \ dx\leq \mathcal{C}.
\end{split}
\end{equation}
In other words, we get the following estimate:
\begin{equation}\label{ia_sa_pos}
    \int_\Omega {a_\varepsilon(|\nabla u_\varepsilon|)}^{2k} |D^2u_\varepsilon|^2 \ dx\leq
    \mathcal{C}, 
\end{equation}
for any $k \geq \frac{s_a-\alpha}{2s_a}$.\\
Moreover, by \eqref{estimate_a_eps} and \eqref{linfinito}, we deduce 
\begin{equation}\label{L2primocas}
    {a_\varepsilon(|\nabla u_\varepsilon|)}^{k}|\nabla u_\varepsilon| \leq C(i_a,s_a,k,\| \nabla u_\varepsilon\|_{L^\infty(\Omega)}),
\end{equation}
for $k\geq -1/s_a$. In addition, for every $\alpha\in [0,1)$, we note that $(s_a-\alpha)/2s_a>-1/s_a$.\\
\textbf{Case 2.} If $i_a < s_a\leq0$, by \eqref{estimate_a_eps}, we have:
\begin{equation*}
    \frac{1}{(\varepsilon + |\nabla u_\varepsilon|^2)^{\frac{\alpha}{2}}} \geq \frac{a(1)^{\frac{\alpha}{i_a}}}{a_\varepsilon(|\nabla u_\varepsilon|)^{\frac{\alpha}{i_a}}}.
\end{equation*}
Therefore, by \eqref{second_deriv_eps},
\begin{equation}\label{int_n}
    \int_\Omega a_\varepsilon(|\nabla u_\varepsilon|)^{1-\frac{\alpha}{i_a}} |D^2u_\varepsilon|^2 \ dx\leq \mathcal{C}.
\end{equation}
Let $\gamma \leq 0$. By \eqref{estimate_a_eps} and \eqref{linfinito} we get:
\begin{equation}\label{bef_g}
    \frac{a(1)^\gamma}{a_\varepsilon(|\nabla u_\varepsilon|)^\gamma} \geq \frac{1}{(\varepsilon + |\nabla u_\varepsilon|^2)^{\frac{s_a}{2}\gamma}} \geq C(\|\nabla u_\varepsilon\|_{L^\infty(\Omega)},s_a,\gamma),
\end{equation}
By \eqref{int_n} and \eqref{bef_g}, for every $\gamma\leq 0$, we have
\begin{equation}\label{after_gamma_2}
\begin{split}
C(\|\nabla u_\varepsilon\|_{L^\infty(\Omega)},s_a,\gamma)&\int_\Omega {a_\varepsilon(|\nabla u_\varepsilon|)^{1-\frac{\alpha}{i_a}+ \gamma}} |D^2u_\varepsilon|^2 \ dx \\&\leq\int_\Omega \frac{a_\varepsilon(|\nabla u_\varepsilon|)^{1-\frac{\alpha}{i_a}+ \gamma}}{a_\varepsilon(|\nabla u_\varepsilon|)^\gamma} |D^2u_\varepsilon|^2 \ dx\leq \mathcal{C}.
\end{split}
\end{equation}
Hence, we obtain the following estimate
\begin{equation}\label{ia_sa_neg}
    \int_\Omega {a_\varepsilon(|\nabla u_\varepsilon|)}^{2k} |D^2u_\varepsilon|^2 \ dx\leq \mathcal{C}, 
\end{equation}
for $k \leq \frac{i_a-\alpha}{2i_a}$.\\
Furthermore, by \eqref{estimate_a_eps} and \eqref{linfinito}, we deduce 
\begin{equation}\label{L2seccas}
    {a_\varepsilon(|\nabla u_\varepsilon|)}^{k}|\nabla u_\varepsilon| \leq C(i_a,s_a,k,\| \nabla u_\varepsilon\|_{L^\infty(\Omega)}),
\end{equation}
for $k\leq -1/i_a$. In addition, for every $\alpha\in [0,1)$, we note that $(i_a-\alpha)/2i_a < -1/i_a$.\\
\textbf{Case 3.} If $i_a < 0 <s_a$, we proceed in the following way,
\begin{equation*}
\begin{split}
    \int_\Omega &{a_\varepsilon(|\nabla u_\varepsilon|)}^{2k} |D^2u_\varepsilon|^2 \ dx \\
    &\quad = \int_\Omega {a_\varepsilon(|\nabla u_\varepsilon|)}^{2k-1} (\varepsilon + |\nabla u_\varepsilon|^2)^{\frac{\alpha}{2}} \frac{a_\varepsilon(|\nabla u_\varepsilon|)}{(\varepsilon + |\nabla u_\varepsilon|^2)^{\frac{\alpha}{2}}}|D^2u_\varepsilon|^2 \ dx. 
\end{split}
\end{equation*}
Let us now distinguish two cases, namely:\\
    i) If $2k -1 \geq 0$, by \eqref{estimate_a_eps} it follows that:
    \begin{equation*}
        a_\varepsilon(|\nabla u_\varepsilon|)^{2k-1} \leq a(1)^{2k-1}(\varepsilon + |\nabla u_\varepsilon|^2)^{\frac{i_a}{2}(2k-1)}.
    \end{equation*}
    Therefore, by \eqref{linfinito} and \eqref{second_deriv_eps},
    \begin{equation*}
    \begin{split}
        \int_\Omega &{a_\varepsilon(|\nabla u_\varepsilon|)}^{2k} |D^2u_\varepsilon|^2 \ dx \\
        &\quad\leq a(1)^{2k-1}\int_\Omega (\varepsilon + |\nabla u_\varepsilon|^2)^{\frac{i_a}{2}(2k-1)+\frac{\alpha}{2}}\frac{a_\varepsilon(|\nabla u_\varepsilon|)}{(\varepsilon + |\nabla u_\varepsilon|^2)^{\frac{\alpha}{2}}}|D^2u_\varepsilon|^2 \ dx \leq \mathcal{C},
    \end{split}
    \end{equation*}
    for any $k \leq \frac{1}{2} - \frac{\alpha}{2i_a}$.\\
    ii) If instead, $2k -1 < 0$, we have the estimate:
    \begin{equation*}
         a_\varepsilon(|\nabla u_\varepsilon|)^{2k-1} \leq a(1)^{2k-1}(\varepsilon + |\nabla u_\varepsilon|^2)^{\frac{s_a}{2}(2k-1)}.
    \end{equation*}
    Hence, again by \eqref{linfinito} and \eqref{second_deriv_eps},
    \begin{equation*}
    \begin{split}
        \int_\Omega &{a_\varepsilon(|\nabla u_\varepsilon|)}^{2k} |D^2u_\varepsilon|^2 \ dx \\
        &\quad\leq \int_\Omega (\varepsilon + |\nabla u_\varepsilon|^2)^{\frac{s_a}{2}(2k-1)+\frac{\alpha}{2}}\frac{a_\varepsilon(|\nabla u_\varepsilon|)}{(\varepsilon + |\nabla u_\varepsilon|^2)^{\frac{\alpha}{2}}}|D^2u_\varepsilon|^2 \ dx \leq \mathcal{C},
    \end{split}
    \end{equation*}
    for any $k \geq \frac{1}{2} - \frac{\alpha}{2s_a}$.\\
Summing up, we can conclude that:
\begin{equation}\label{ia_neg_sa_pos}
    \int_\Omega {a_\varepsilon(|\nabla u_\varepsilon|)}^{2k} |D^2u_\varepsilon|^2 \ dx\leq \mathcal{C}, 
\end{equation}
for $\frac{s_a-\alpha}{2s_a}\leq k \leq \frac{i_a-\alpha}{2i_a}$.\\
Additionally, by \eqref{estimate_a_eps} and \eqref{linfinito}, we deduce 
\begin{equation}\label{L2terzcas}
    {a_\varepsilon(|\nabla u_\varepsilon|)}^{k}|\nabla u_\varepsilon| \leq C(i_a,s_a,k,\| \nabla u_\varepsilon\|_{L^\infty(\Omega)}),
\end{equation}
for $-1/s_a \leq k\leq -1/i_a$. In addition, for every $\alpha\in [0,1)$, we note that $(i_a-\alpha)/2i_a < -1/i_a$ and $(s_a-\alpha)/2s_a>-1/s_a$.

\vspace{0.3 cm}
\noindent
A direct computation shows that:
\begin{equation}\label{W12ineq}
\begin{split}
    &\left |\nabla (a_\varepsilon(|\nabla u_\varepsilon|)^{k}u_{\varepsilon,i}) \right | \\
    &\ \ \leq C(n)\left[|k|\max\{|s_a|,|i_a|\} +1\right] a_\varepsilon(|\nabla u_\varepsilon|)^{k}| D^2  u_\varepsilon | \quad \text{a.e. in } \Omega,
\end{split}
\end{equation}
for $i=1,...,n$ and for some positive constant $C(n)$.\\
Therefore, in light of the preceding inequalities \eqref{ia_sa_pos}, \eqref{L2primocas}, \eqref{ia_sa_neg}, \eqref{L2seccas}, \eqref{ia_neg_sa_pos}, \eqref{L2terzcas} and \eqref{W12ineq} we have:
\begin{equation}\label{bound_quasi_stress_field}
\left \|a_\varepsilon(|\nabla u_\varepsilon|)^{k}\nabla u_{\varepsilon} \right \|_{W^{1,2}(\Omega)}\leq C,
\end{equation}
where $C=C(n,i_a,s_a,\alpha,L_\Omega,d_\Omega,\Psi_{\Omega},|\Omega|,\|\operatorname{tr}\mathcal{B}\|_{L^{n-1,1}(\Omega)},\| f\|_{L(\Omega)},\| f\|_{W^{1,1}(\Omega)})$ is a positive constant and $k$ in the corresponding range of \textbf{Case 1}, \textbf{Case 2} and \textbf{Case 3}. Notice that the dependence of the constant by $\Psi_\Omega$ is due to Lemma \ref{antoninho2}.\\
Owing to the uniform boundedness in $W^{1,2}$, there exists a function $V \in W^{1,2}(\Omega)$ such that, up to a subsequence $\{\varepsilon_l\}$,
\begin{equation*}
    a_{\varepsilon_l}(|\nabla u_{\varepsilon_l}|)^k\nabla u_{\varepsilon_l} \rightarrow  V \text{ in } L^2(\Omega),
\end{equation*}
\begin{equation}\label{weakconv}
     a_{\varepsilon_l}(|\nabla u_{\varepsilon_l}|)^k\nabla u_{\varepsilon_l} \rightharpoonup  V \text{ in } W^{1,2}(\Omega).
\end{equation}
In addition, by \cite[Theorem $1.7$]{Li2}, $u_{\varepsilon_l}$ is uniformly bounded in $C^{1,\alpha}$, for any open set $\Omega' \subset\subset\Omega$. Therefore, there exists $w\in C^1(\Omega')$ such that:
\begin{equation*}
    u_{\varepsilon_l} \rightarrow w \quad \text{in } C^{1,\alpha'}(\Omega'),
\end{equation*}
with $\alpha' < \alpha$.\\
We claim that:
\begin{equation}\label{claim_A}
    a_{\varepsilon_l}(|\nabla u_{\varepsilon_l}|) \rightarrow a(|\nabla w|), 
\end{equation}
uniformly in compact subsets of $\Omega \setminus\{\nabla w = 0\}$.\\
Indeed, by the uniform convergence, there exists $\varepsilon_{l,0}>0$ such that for every $\varepsilon_l \leq \varepsilon_{l,0}$, $\nabla u_{\varepsilon_l} \neq 0$, for any $x \in \mathcal{I}(x_0)$, where $\mathcal{I}(x_0)$ denotes a neighborhood of $x_0 \in \Omega \setminus\{\nabla w = 0\}$.

Moreover:
\begin{equation*}
\begin{split}
    &a_{\varepsilon_l}(|\nabla u_{\varepsilon_l}|) - a(|\nabla w|)\\
    &\ = a_{\varepsilon_l}(|\nabla u_{\varepsilon_l}|) - a_{\varepsilon_l}(|\nabla w|) + a_{\varepsilon_l}(|\nabla w|)- a(|\nabla w|)\\
    &\ = a'_{\varepsilon_l}(\xi)(|\nabla u _{\varepsilon_l}|-|\nabla w|) + a_{\varepsilon_l}(|\nabla w|)- a(|\nabla w|),\\
\end{split}
\end{equation*}
where $\xi \in (\min\{|\nabla u_{\varepsilon_l}|,|\nabla w|\},\max\{|\nabla u_{\varepsilon_l}|,|\nabla w|\})$.\\
Hence, the claim \eqref{claim_A} follows by the uniform convergence of the gradients and by the fact that, where $\nabla u_{\varepsilon_l} \neq 0$ and $\nabla w \neq 0$, $a'_{\varepsilon_l}(\xi)$ is uniformly bounded and $a_{\varepsilon_l}(|\nabla w|) \rightarrow a(|\nabla w|)$ as $l \rightarrow \infty$, by \eqref{conv_a_eps}.\\
In particular,
\begin{equation*}
    a(|\nabla w|)^{k}\nabla w =  V, \quad \text{ for any } x\in \Omega \setminus \{\nabla w = 0\}.
\end{equation*}
On the other hand, in the critical set $\{\nabla w=0\}$, with $k$ in the threshold given by \eqref{T_k}, we have
\begin{equation*}
    a(|\nabla w|)^{k}\nabla w = 0,
\end{equation*}
and,
\begin{equation*}
    a_{\varepsilon_l}(|\nabla u_{\varepsilon_l}|)^{k}\nabla u_{\varepsilon_l} \rightarrow 0, \quad \text{as } l \rightarrow \infty,
\end{equation*}
by the uniform convergence of the gradients and by \eqref{unif_conv}.

Thus, we can conclude that:
\begin{equation}\label{stress_w}
    a(|\nabla w|)^{k}\nabla w =  V \in W^{1,2}(\Omega), \quad \text{ for any } x\in \Omega.
\end{equation}
Testing \eqref{system33} with any test function $\varphi \in C_0^\infty(\Omega)$ leads to:
\begin{equation*}
    \int_\Omega a_{\varepsilon_l}(|\nabla u_{\varepsilon_l}|)\nabla u_{\varepsilon_l} \cdot \nabla \varphi \ dx= \int_\Omega f \ \varphi \ dx.  
\end{equation*}
Passing to the limit as $l \rightarrow \infty$, in view of \eqref{weakconv} and \eqref{stress_w}, yields
\begin{equation}\label{weak_Inf}
    \int_\Omega a(|\nabla w|)\nabla w \cdot \nabla \varphi \ dx= \int_\Omega f \ \varphi \ dx.  
\end{equation}
By \cite[Proposition $2.14$]{Cma_S} and by \eqref{Poincare}, the sequence $u_{\varepsilon_l}$ is bounded in the Orlicz-Sobolev space $W^{1,B_{\varepsilon_l}}_0(\Omega)$, thus, by \eqref{unif_conv}, the function $w$ belongs to $W^{1,B}_0(\Omega)$. This implies that $a(|\nabla w|)\nabla w \in L^{\tilde B}(\Omega)$. Moreover, by Theorem \ref{imbedding}, the space $C_0^\infty(\Omega)$ is dense in $W_0^{1,B}(\Omega)$, thus \eqref{weak_Inf} holds for every $\varphi \in W_0^{1,B}(\Omega)$. In addition, we can conclude, by \cite[Theorem $2.13$]{Cma_S}, that $w=u$, the unique weak solution of \eqref{Dir_pb}.

\vspace{0.5cm}

\textit{Step $2$.} Here, we remove assumption \eqref{bordoregolare}.\\
We consider a sequence of open sets $\{\Omega_m\}$ approximating $\Omega$ in the sense of Lemma \ref{approxcap}. In addition, see \cite{CiaMa}, this sequence is such that:
\begin{equation*}
    \sup_{x\in \partial \Omega} \|\mathcal{B}_m\|_{L^{n-1,\infty}(\partial \Omega_m \cap B_r(x))} \leq C \sup_{x\in \partial \Omega} \|\mathcal{B}\|_{L^{n-1,\infty}(\partial \Omega \cap B_r(x))} \quad \text{if } \ n\geq 3,
\end{equation*}
and,
\begin{equation*}
    \sup_{x\in \partial \Omega} \|\mathcal{B}_m\|_{L^{1,\infty}\log L(\partial \Omega_m \cap B_r(x))} \leq C \sup_{x\in \partial \Omega} \|\mathcal{B}\|_{L^{1,\infty}\log L(\partial \Omega \cap B_r(x))} \quad \text{if } \ n = 2,
\end{equation*}
for some constant $C=C(\Omega)$, where $\mathcal{B}_m$ denotes the second fundamental form on $\partial \Omega_m$.\\
For $m\in \N$, let ${u}_m$ be the weak solution to the Dirichlet problem 
\beq\label{system34}\tag{$p$-$D_m$}
\begin{cases}
-\operatorname{ div}(a(|\nabla u_m|)\nabla{ u_m})=   { f}_m(x)& \mbox{in $\Omega_m$}\\
 u_m= 0  & \mbox{on  $\partial \Omega_m$},
\end{cases}
\eeq
where $ f_m\in C^{\infty}_c(\R^n)$. Here, ${f}_m$
  is constructed by extending ${f}$ to the whole $\R^n$, since $f$ belongs to the Sobolev space, and then using classical mollifiers and standard cut-off functions. With this choice $ f_m$ restricted to $\Omega$ converges to $ f$ in the norm $W^{1,1}(\Omega)\cap L(\Omega).$ Moreover $$\|{f}_m\|_{W^{1,1}(\Omega_m)}\leq C\|{f}\|_{W^{1,1}(\Omega)}\quad\text{and}\quad \|{f}_m\|_{L(\Omega_m)}\leq C \|{f}\|_{L(\Omega)},$$
where $C$ is a positive constant independent on $m$.

By inequality \eqref{bound_quasi_stress_field}, applied to ${u}_m$, and by \eqref{stress_w} we get
\begin{equation}\label{stress_field_u_m}
\left \|a(|\nabla u_m|)^{k}\nabla u_m \right \|_{W^{1,2}(\Omega)}\leq \left \|a(|\nabla u_m|)^{k}\nabla u_m  \right \|_{W^{1,2}(\Omega_m)} \leq C,
\end{equation}
where $C(n,i_a,s_a,\alpha,L_\Omega,d_\Omega,\Psi_{\Omega},\|\operatorname{tr}\mathcal{B}\|_{L^{n-1,1}(\Omega)},\|  f\|_{L(\Omega)},\| f\|_{W^{1,1}(\Omega)})$ is a positive constant independent on $m$ and $k$ in the corresponding range. Note that this dependence of the constant $\mathcal{C}$ is guaranteed by the properties \eqref{prop1} and \eqref{prop2} of the sequence $\{\Omega_m\}$ and by the convergence of $ f_m$ to ${f}$.

By \eqref{stress_field_u_m}, we deduce that there exists ${V} \in W^{1,2}(\Omega)$ such that, up to a subsequence still denoted by $\{u_m\}$,
\begin{equation}\label{weakconv_m}
\begin{split}
    &a(|\nabla u_m|)^{k}\nabla u_m \rightarrow  V \text{ in } L^2(\Omega),\\
    &a(|\nabla u_m|)^{k}\nabla u_m \rightharpoonup  V \text{ in } W^{1,2}(\Omega),
    \end{split}
\end{equation}

As in \textit{Step 1}, by a result in \cite{Li2}, there exist $\alpha \in (0,1)$ and a positive constant independent on $m$, such that:
\begin{equation*}
    \|u_m\|_{C^{1,\alpha}(\Omega')} \leq C,
\end{equation*}
for every open set $\Omega' \subset \subset \Omega$. In particular,
\begin{equation*}
    u_m \rightarrow v \quad \text{in } C^{1,\alpha'}(\Omega'),
\end{equation*}
with $\alpha' < \alpha$, for some function $v \in C^1(\Omega')$.\\
Repeating the same procedure as in \textit{Step 1}, leads to:
\begin{equation}\label{stress_v}
    a(|\nabla v|)^{k}\nabla v = V \in W^{1,2}(\Omega).
\end{equation}
Given as test function $ \varphi \in C_0^\infty(\Omega)$ in the weak formulation of \eqref{system34}, and pass to the limit as $m \rightarrow \infty$ in the resulting equation, by \eqref{weakconv_m} and \eqref{stress_v}, we obtain:
\begin{equation}\label{limit_Dvs}
    \int_{\Omega}a(|\nabla v|)\nabla v\cdot \nabla{\varphi}\,dx=\int_\Omega f\varphi\,dx.
\end{equation}
The exact same argument previously applied to the equation \eqref{weak_Inf} ensures that equation \eqref{limit_Dvs} continues to hold for any test function $\varphi \in W^{1,B}_0(\Omega)$. Hence, $v=u$, by the uniqueness of \eqref{Dir_pb} and,
\begin{equation*}
    a(|\nabla u|)^{k}\nabla u \in W^{1,2}(\Omega).
\end{equation*}
%

\textit{Step $3$.} The last step consists in removing the assumption \eqref{reg_f}. Let ${f} \in W^{1,1}(\Omega) \cap L(\Omega)$. By standard density argument one can infer that there exists a sequence $\{{f}_s\} \subset C^\infty(\overline\Omega)$ (as in \textit{Step 2}) such that
\begin{equation}\label{conv_f}
    {f}_s \rightarrow {f} \quad \text{in } W^{1,1}(\Omega) \cap L(\Omega).
\end{equation}

To go ahead we consider a sequence $\{{u}_s\}$ of weak solutions to the following Dirichlet problem:
\beq\label{system_step3}
\begin{cases}
-\operatorname{ div}(a(|\nabla u_s|)\nabla{ u_s})=   { f_s}& \mbox{in $\Omega$}\\
 u_s= 0  & \mbox{on  $\partial \Omega$}.
\end{cases}
\eeq
By inequality \eqref{stress_field_u_m} of \textit{Step $2$} applied to $ u_s$ we deduce:
\begin{equation}\label{stress_field_u_k}
    \begin{split}
&\left \|a(|\nabla u_s|)^{k}\nabla u_s \right \|_{W^{1,2}(\Omega)}\\
&\quad\leq C(n,i_a,s_a,\alpha,L_\Omega,d_\Omega,\Psi_{\Omega},\|\operatorname{tr}\mathcal{B}\|_{L^{n-1,1}(\Omega)},\|  f_s\|_{L(\Omega)},\| f_s\|_{W^{1,1}(\Omega)})\\
&\quad\leq \bar C(n,i_a,s_a,\alpha,L_\Omega,d_\Omega,\Psi_{\Omega},\|\operatorname{tr}\mathcal{B}\|_{L^{n-1,1}(\Omega)},\| f\|_{L(\Omega)},\| f\|_{W^{1,1}(\Omega)}),
\end{split}
\end{equation}
where the last inequality follows by \eqref{conv_f}.\\
By \eqref{stress_field_u_k}, there exists $W\in W^{1,2}(\Omega)$ such that:
\begin{equation}\label{conv_2}
a(|\nabla u_s|)^{k}\nabla u_s \rightarrow W \quad \text{in } L^2(\Omega), \quad a(|\nabla u_s|)^{k}\nabla u_s \rightharpoonup W \quad \text{in } W^{1,2}(\Omega).
\end{equation}

By \cite[Proposition $2.14$]{Cma_S}, there exists a positive constant $C$ independent on $s$ such that:
\begin{equation}\label{bound_u_k}
\begin{split}
    \|\nabla u_s\|_{L^B(\Omega)}& \leq C(n,i_a,s_a,|\Omega|,\|f_s\|_{L(\Omega)})\\
    &\leq C(n,i_a,s_a,|\Omega|,\|f\|_{L(\Omega)}),
\end{split}
\end{equation}
where we used \eqref{conv_f}.\\
In addition, by \eqref{Poincare}, the sequence $\{u_s\}$ is bounded in $W^{1,B}_0(\Omega)$. Since $W_0^{1,B}(\Omega)$ is reflexive and the embedding of $W^{1,B}_0(\Omega)$ on $L^B(\Omega)$ is compact, there exists a function $u \in W^{1,B}_0(\Omega)$ such that, up to a subsequence,
\begin{equation*}
    u_s \rightarrow u \quad \text{in } L^B(\Omega),
\end{equation*}
\begin{equation}\label{W1_B_conv}
    u_s \rightharpoonup u \quad \text{in } W^{1,B}_0(\Omega),
\end{equation}
We claim that:
\begin{equation}\label{convergenzalp}
    \nabla{u}_s \rightarrow \nabla{u} \quad \text{a.e. in } \Omega.
\end{equation}
Indeed, testing \eqref{system_step3} and its analogue with $s$ replaced by $m$, by $u_s - u_m \in W^{1,B}_0(\Omega)$, and subtracting the resulting equations, we get:
\begin{equation*}
\begin{split}
    \int_\Omega &(a(|\nabla u_s|)\nabla u_s - a(|\nabla u_m|)\nabla u_m)\cdot \nabla(u_s-u_m) \ dx\\
    &\qquad \qquad= \int_{\Omega}(f_s-f_m)(u_s - u_m) \ dx\\
    &\qquad \qquad\leq \|f_s-f_m\|_{L^{\tilde B_n}(\Omega)}\|u_s-u_m\|_{L^{B_n}(\Omega)}\\
    &\qquad \qquad\leq C(\Omega,B)\|f_s-f_m\|_{L^{\tilde B_n}(\Omega)}\|\nabla(u_s-u_m)\|_{L^{B}(\Omega)}\\
    &\qquad \qquad\leq \tilde C(\Omega,B)\|f_s-f_m\|_{L^{n,1}(\Omega)}\|\nabla(u_s-u_m)\|_{L^{B}(\Omega)}\\
    &\qquad \qquad\leq \tilde C(\Omega,B)\|f_s-f_m\|_{L^{n,1}(\Omega)}\left(\|\nabla u_s\|_{L^{B}(\Omega)}+\|\nabla u_m\|_{L^{B}(\Omega)}\right),
\end{split}
\end{equation*}
notice that in the previous inequalities we have used \eqref{Sobolev} and the embedding of $L^{n,1}$ into $L^{\tilde B_n}$ (see \cite[Remark $2.12$]{CiaMa}).\\
Hence, using \eqref{bound_u_k} and \eqref{conv_f}, we get

\begin{equation*}
    \lim_{s \rightarrow\infty, m \rightarrow\infty} \int_\Omega (a(|\nabla u_s|)\nabla u_s - a(|\nabla u_m|)\nabla u_m)\cdot \nabla(u_s-u_m) \ dx =0.
\end{equation*}
We shall now follow an argument from \cite{Boccardo} (see also \cite[Step 4]{Cma_S}. Let $\delta > 0$. Given any $t,\tau>0$, we have:
\begin{equation*}
\begin{split}
    &|\{|\nabla u_s - \nabla u_m|>t\}|\\
    &\quad\leq |\{|\nabla u_s|>\tau\}|+|\{|\nabla u_m|>\tau\}|+|\{|\nabla u_s - \nabla u_m|>t,|\nabla u_s|\leq\tau,|\nabla u_m|\leq\tau\}|,
\end{split}
\end{equation*}
for $s,m \in \mathbb{N}$. Moreover, by \eqref{bound_u_k}, if $\tau$ is sufficiently large,
\begin{equation*}
    |\{|\nabla u_s|>\tau\}| < \delta, \quad \text{for } s \in \mathbb{N}.
\end{equation*}
Let us define:
\begin{equation*}
    \theta = \inf\{[a(|\xi|)\xi-a(|\eta|)\eta]\cdot (\xi-\eta) : |\xi-\eta|>t, |\xi| \leq \tau, |\eta|\leq \tau\},
\end{equation*}
which is strictly positive. Thus,
\begin{equation*}
\begin{split}
\theta &|\{|\nabla u_s - \nabla u_m|>t,|\nabla u_s|\leq\tau,|\nabla u_m|\leq\tau\}|\\
    &\qquad\leq\int_\Omega (a(|\nabla u_s|)\nabla u_s - a(|\nabla u_m|)\nabla u_m)\cdot \nabla(u_s-u_m) \ dx
\end{split}
\end{equation*}
Therefore, for $s$ and $m$ sufficiently large,
\begin{equation*}
    |\{|\nabla u_s - \nabla u_m|>t,|\nabla u_s|\leq\tau,|\nabla u_m|\leq\tau\}| < \delta.
\end{equation*}
Summing up, $\nabla u_s$ is a Cauchy sequence in measure, thus, up to a subsequence, the claim follows. Therefore, proceeding in a similar way as in \textit{Step $1$}, we obtain:
\begin{equation}\label{conv_a_e}
    a(|\nabla u_s|)^{k}\nabla u_s \rightarrow a(|\nabla u|)^{k}\nabla u \quad \text{a.e. in } \Omega.
\end{equation}
By \eqref{conv_a_e} and \eqref{conv_2} one can pass to the limit as $s \rightarrow \infty$ in the weak formulation of \eqref{system_step3}, and deduce that $u$ is the unique solution of \eqref{Dir_pb}. Notice that the uniqueness result for solutions to \eqref{Dir_pb} is provided by \cite[Theorem $2.13$]{CiaMa}.
Hence,
\begin{equation}\label{Fine_D}
W= a(|\nabla u|)^{k}\nabla u \in W^{1,2}(\Omega).
\end{equation}
\vspace{0.5 cm}

\textbf{Neumann case.} Similarly, we can prove the same results for the Neumann problem. In what follows, we mention the main changes required in each step.

\textit{Step $1.$} In this framework we take $ u_\varepsilon$, normalized by its mean, as the weak solution of the following Neumann problem:
\beq\label{system33N}\tag{$p$-$N_\varepsilon$}
\begin{cases}
-\operatorname{ div}(a_\varepsilon(|\nabla u_\varepsilon|)\nabla{ u_\varepsilon})=   { f}(x)& \mbox{in $\Omega$}\\
{\frac{\partial  u_\varepsilon}{\partial  \nu}} = 0  & \mbox{on  $\partial \Omega$}.
\end{cases}
\eeq
We recover the $L^\infty$-estimate by \cite[Theorem $1.1$]{Cma_S} (see also  \cite{Cma2}). Moreover, by classical regularity $u_\varepsilon \in W^{2,2}(\Omega)$. In addition, as in the Dirichlet case, an iterative procedure yields $u_\varepsilon \in C^3(\overline \Omega)$. Thus, we can apply the estimate \eqref{Neu_estimate} to $u_\varepsilon$ solution of \eqref{system33N}:
\begin{equation}\label{Neu_estimate}
\begin{split}
&\int_\Omega \frac{a_\varepsilon(|\nabla u_\varepsilon|)}{(\varepsilon + |\nabla u_\varepsilon|^2)^{\frac{\alpha}{2}}} |D^2u_\varepsilon|^2 \ dx\\
&\qquad \leq \mathcal{C} \left (\int_{\Omega} \frac{a_\varepsilon(|\nabla u_\varepsilon|) }{(\varepsilon + |\nabla u_\varepsilon|^2)^{\frac{\alpha}{2}}} |\nabla u_\varepsilon|^2 \ dx \right.\\
&\left.\qquad \qquad + \sum_{i=1}^n \int_\Omega \operatorname{div}(a_\varepsilon(|\nabla u_\varepsilon|)\nabla u_\varepsilon) \partial_{x_i} \left(\frac{u_{\varepsilon,i}}{(\varepsilon + |\nabla u_\varepsilon|^2)^{\frac{\alpha}{2}}}\right)\ dx\right).
\end{split}
\end{equation}
An integration by parts, since $\frac{\partial u_\varepsilon}{\partial \nu}=0$ on $\partial \Omega$, we deduce:
\begin{equation*}
\begin{split}
&\int_\Omega \frac{a_\varepsilon(|\nabla u_\varepsilon|)}{(\varepsilon + |\nabla u_\varepsilon|^2)^{\frac{\alpha}{2}}} |D^2u_\varepsilon|^2 \ dx\\
&\qquad \leq \mathcal{C} \left (\int_{\Omega} \frac{a_\varepsilon(|\nabla u_\varepsilon|) }{(\varepsilon + |\nabla u_\varepsilon|^2)^{\frac{\alpha}{2}}} |\nabla u_\varepsilon|^2 \ dx \right.\\
&\left.\qquad \qquad - \sum_{i=1}^n \int_\Omega \partial_{x_i}\operatorname{div}(a_\varepsilon(|\nabla u_\varepsilon|)\nabla u_\varepsilon) \frac{u_{\varepsilon,i}}{(\varepsilon + |\nabla u_\varepsilon|^2)^{\frac{\alpha}{2}}}\ dx\right).
\end{split}
\end{equation*}

With these variants, the remaining part follows analogously as in the Dirichlet case.

\textit{Step $2$.} Instead of \eqref{system34} we consider a sequence of weak solutions $\{ u_m\}$ of:
\beq\label{system34N}\tag{$p$-$N_m$}
\begin{cases}
-\operatorname{ div}(a(|\nabla u_m|)\nabla{ u_m})=   { f}_m(x)& \mbox{in $\Omega_m$}\\
{\frac{\partial  u_m}{\partial  \nu}}= 0  & \mbox{on  $\partial \Omega_m$},
\end{cases}
\eeq
Since $u_m$ is uniformly bounded in $C^{1,\alpha}(\Omega')$, for every $\Omega' \subset \subset \Omega$, there exists a function $u \in C^{1,\alpha}(\Omega')$ such that:
\begin{equation*}
u_m \rightarrow u \quad \text{and} \quad \nabla u_m \rightarrow \nabla u \quad \text{uniformly in } \Omega'.
\end{equation*}
Thus, passing to the limit in the weak formulation of \eqref{system34N}, we get:
\begin{equation}\label{NEU_2}
    \int_{\Omega}a(|\nabla u|)\nabla u\cdot \nabla{\psi}\,dx=\int_{\Omega} f\psi\,dx,
\end{equation}
for any ${\psi}\in W^{1,\infty}(\Omega)$, since any function from $ W^{1,\infty}(\Omega)$ can be extended to a function in $W^{1,\infty}(\R^n)$. Moreover, by density, \eqref{NEU_2} holds for any $\psi \in W^{1,B}(\Omega)$ and the remaining part follows as in the Dirichlet case.

\textit{Step $3$.} The last step is entirely analogous, taking as $\{u_s\}$ a sequence of weak solutions of the problem:
\begin{equation*}
    \begin{cases}
-\operatorname{div}(a(\nabla u_s)\nabla u_s)=   { f_s}& \mbox{in $\Omega$}\\
\frac{\partial u_s}{\partial \nu}= 0  & \mbox{on  $\partial \Omega$}.
\end{cases}
\end{equation*}

\vspace{0.5 cm}

\end{proof}

\begin{proof}[Proof of Corollary \ref{CorW22}.]
The thesis follows by an application of Theorem \ref{teo1INTRO} with $k =0$.
\end{proof}

\vspace{0.3 cm}

We conclude this section by analyzing the special case of convex domains. 

\begin{proof}[Proof of Theorem \ref{conv_d}]
    The proof is similar to that of Theorem \ref{teo1INTRO}. The only difference relies on choosing, in \textit{Step $2$}, a sequence $\Omega_m$ of open convex bounded sets approximating $\Omega$ from outside with respect to the Hausdorff distance. Furthermore, conditions \eqref{prop1} are automatically fulfilled with convex domains and condition \eqref{prop2} does not play any role, since the constant $\mathcal{C}$ in inequalities \eqref{Dir_estimate} and \eqref{Neu_estimate} is independent of $K_\Omega$.
\end{proof}

\begin{remark}\label{misuraZu}
    Notice that since $a(|\nabla u|)\nabla u \in W^{1,2}(\Omega)$, by Stampacchia's theorem (see, for instance, \cite[Theorem $6.19$]{LL}), we have:
    \begin{equation}\label{grad0}
        \partial_{x_i} (a(|\nabla u|)\nabla u) = 0 \quad \text{a.e. in } \{a(|\nabla u|)\nabla u = 0\}.
    \end{equation}
    Moreover, by the properties of $a$ we have that $a(|\nabla u|)\nabla u = 0$ in $Z_u :=\{x \in \Omega : \nabla u =0\}$.\\
    Let us now suppose that $f(x)\neq 0$, $\forall x \in \Omega$, by \cite[Corollary $4.24$]{Brezis},
    \begin{equation*}
        -\operatorname{div}(a(|\nabla u|)\nabla u) = f(x) \quad \text{a.e. in } \Omega.
    \end{equation*}
    Furthermore by \eqref{grad0},
    \begin{equation*}
        -\operatorname{div}(a(|\nabla u|)\nabla u) = 0 \quad \text{a.e. in } Z_u.
    \end{equation*}
    Therefore, we can conclude that the Lebesgue measure of $Z_u$ is zero, namely $|Z_u|=0$.
\end{remark}

\section{Integrability of the inverse of the gradient}\label{integrability_section}

We emphasize that the results established in Section \ref{second-order} remain valid without imposing any sign condition on the source term $f$. However, assuming additionally that $f \geq \tau > 0$ or $f \leq -\tau < 0$ locally, we can infer, as a consequence of Theorem \ref{teo1INTRO}, integrability properties for the inverse of the weight, specifically $a(|\nabla u|)^{-1}$. The following lemma establishes a global integral estimate under the assumption of regularity of some function $u$ and smoothness of the domain $\Omega$:
\begin{lemma}\label{ineq_peso_u}
    Let $\Omega$ be a bounded open set in $\R^n$, with $\partial \Omega \in C^2$. Let $a_\varepsilon$ be defined as in Theorem \ref{STIMA}.\\
    Then, for any $\theta>0$ and $\beta\in \R$ we have:
    \begin{equation}\label{preinversodelpeso}
        \begin{split}
            &\left|\int_\Omega \operatorname{div} \left(a_\varepsilon (|\nabla u|)\nabla u\right)\frac{\psi ^2}{\left(a_\varepsilon(|\nabla u|)\right)^\beta} \,dx\right| \\&\leq \int_\Omega \left(a_\varepsilon(|\nabla u|)\right)^{1-\beta}|\nabla u||\nabla \psi |^2\,dx+ \int_\Omega \left(a_\varepsilon(|\nabla u|)\right)^{1-\beta}|\nabla u| \psi^2\,dx\\&+\theta\beta \int_\Omega \frac{\psi ^2}{\left(a_\varepsilon(|\nabla u|)\right)^\beta}\,dx+s_{a_\varepsilon}^2 \int_\Omega \left(a_\varepsilon(|\nabla u|)\right)^{2-\beta}|D^2u|^2\,dx\\&+\int_{\partial \Omega} \left(a_\varepsilon (|\nabla u|)\right)^{1-\beta} |\nabla u| \psi^2 \,d\mathcal{H}^{n-1}.
        \end{split}
    \end{equation}
    for any vector field $u \in C^3(\Omega) \cap C^2(\bar \Omega)$ under Dirichlet boundary condition, for any $\psi\in C^{\infty}_c(\R^n)$. If instead $u$ obeys to the Neumann boundary condition, \eqref{preinversodelpeso} holds without the boundary term.
\end{lemma}

\begin{proof}
    We prove the inequality \eqref{preinversodelpeso} with the positive sign. The other case is similar. Let $\alpha \in \R$ and define $$\varphi:= \frac{\psi ^2}{\left(a_\varepsilon (|\nabla u|)\right)^\beta},$$ where $\psi \in C^{\infty}_c(\R^N)$. We consider the following equality, obtained by an integration by parts:
    \begin{equation}\label{inequalitya1}
        \begin{split}
            &\int_\Omega \operatorname{div} \left(a_\varepsilon (|\nabla u|)\nabla u\right)\varphi \,dx \\&\quad=-\int_\Omega a_\varepsilon (|\nabla u|)(\nabla u\cdot\nabla\varphi) \,dx+\int_{\partial \Omega} a_\varepsilon (|\nabla u|) (\nabla u\cdot\nu) \varphi \,d\mathcal{H}^{n-1},
        \end{split}
    \end{equation}
where $\nu$ is the outward unit vector on $\partial \Omega$.
We note that, under the Neumann condition, the boundary term in the equality \eqref{inequalitya1} is equal to zero. Now we estimate the first term on the right-hand side of the equality \eqref{inequalitya1}. Using a weighted Young's inequality, we have 
\begin{equation}\label{teminebuio}
    \begin{split}
        &-\int_\Omega a_\varepsilon (|\nabla u|)(\nabla u\cdot\nabla\varphi) \,dx\\&\quad=-2\int_\Omega a_\varepsilon (|\nabla u|)(\nabla u\cdot\nabla\psi)\psi\frac{1}{\left(a_\varepsilon (|\nabla u|)\right)^\beta}\,dx\\&\quad\quad+\beta\int_\Omega a'_\varepsilon(|\nabla u|) \left(\nabla u\cdot \frac{D^2u\nabla u}{|\nabla u|}\right)\frac{\psi ^2}{\left(a_\varepsilon (|\nabla u|)\right)^\beta}\,dx\\&\quad\leq \int_\Omega \left(a_\varepsilon(|\nabla u|)\right)^{1-\beta}|\nabla u||\nabla \psi |^2\,dx+ \int_\Omega \left(a_\varepsilon(|\nabla u|)\right)^{1-\beta}|\nabla u| \psi^2\,dx\\&\quad \quad+\theta\beta \int_\Omega \frac{\psi ^2}{\left(a_\varepsilon(|\nabla u|)\right)^\beta}\,dx+\frac{\beta}{4\theta}\int_\Omega \frac{\left(a'_\varepsilon(|\nabla u|)\right)^2|\nabla u|^2}{\left(a_\varepsilon(|\nabla u|)\right)^\beta}|D^2 u|^2\psi ^2\,dx,
    \end{split}
\end{equation}
where $\theta$ is a positive constant. 
By \eqref{cond_a_hat}, the last term of the previous inequality \eqref{teminebuio} can be estimated in the following way:
\begin{equation}\label{lasttermdiunadisuguaglianza}
    \int_\Omega \frac{\left(a'_\varepsilon(|\nabla u|)\right)^2|\nabla u|^2}{\left(a_\varepsilon(|\nabla u|)\right)^\beta}|D^2 u|^2\psi ^2\,dx \leq s_{a_\varepsilon}^2 \int_\Omega \left(a_\varepsilon(|\nabla u|)\right)^{2-\beta}|D^2u|^2\,dx.
\end{equation}

By \eqref{inequalitya1}, \eqref{teminebuio} and \eqref{lasttermdiunadisuguaglianza} we have 

 \begin{equation}\label{inequalitya11}
        \begin{split}
            &\int_\Omega \operatorname{div} \left(a_\varepsilon (|\nabla u|)\nabla u\right)\varphi \,dx \\&\quad\leq \int_\Omega \left(a_\varepsilon(|\nabla u|)\right)^{1-\beta}|\nabla u||\nabla \psi |^2\,dx+ \int_\Omega \left(a_\varepsilon(|\nabla u|)\right)^{1-\beta}|\nabla u| \psi^2\,dx\\&\quad\quad+\theta\beta \int_\Omega \frac{\psi ^2}{\left(a_\varepsilon(|\nabla u|)\right)^\beta}\,dx+s_{a_\varepsilon}^2 \int_\Omega \left(a_\varepsilon(|\nabla u|)\right)^{2-\beta}|D^2u|^2\,dx\\&\quad\quad+\int_{\partial \Omega} a_\varepsilon (|\nabla u|) (\nabla u\cdot\nu) \varphi \,d\mathcal{H}^{n-1}.
        \end{split}
    \end{equation}

    \end{proof}

We are now in position to prove Theorem \ref{peso_stima_Intro}.

\begin{proof}[Proof of Theorem \ref{peso_stima_Intro}]
The conclusion is ultimately derived from Lemma \ref{ineq_peso_u} together with Theorem \ref{teo1INTRO}. However, as in the proof of Theorem \ref{teo1INTRO}, the application of the lemma requires regularization steps, that involve both the domain $\Omega$ and the source term ${f}$. For this reason, we structure the proof in several distinct steps. We prove Theorem \ref{peso_stima_Intro} when $f$ has a positive sign, the other case is similar.

\textit{Step $1.$} Here, we assume 
\begin{equation}\label{bordoregolare_P}
   \partial \Omega\in C^\infty,
\end{equation}
\vspace{0.05cm}
\begin{equation}\label{reg_f_P}
    f \in C^\infty(\Omega).
\end{equation}
Given $\varepsilon$ small enough, we consider a weak solution ${u_\varepsilon}$ to the approximating problem:
\begin{equation}\label{eq:regolarizz}
    -\operatorname{ div}( a_\varepsilon(|\nabla u_\varepsilon|)\nabla{ u_\varepsilon})=f,
\end{equation}
under Dirichlet or Neumann boundary conditions. As in the proof of Theorem \ref{teo1INTRO}, the solutions $u_\varepsilon$ of the regularized problem are of class $C^3(\overline \Omega)$. This allows us to apply Lemma \ref{ineq_peso_u} to the approximate solution $u_\varepsilon$, thus obtaining the following: 
\begin{equation}\label{preinversodelpesoapplicatoauepsilon}
        \begin{split}
            &-\int_\Omega \operatorname{div} \left(a_\varepsilon (|\nabla u_\varepsilon|)\nabla u_\varepsilon\right)\frac{\psi ^2}{\left(a_\varepsilon(|\nabla u|)\right)^\beta} \,dx \\&\quad\leq \int_\Omega \left(a_\varepsilon(|\nabla u_\varepsilon|)\right)^{1-\beta}|\nabla u_\varepsilon||\nabla \psi |^2\,dx+ \int_\Omega \left(a_\varepsilon(|\nabla u_\varepsilon|)\right)^{1-\beta}|\nabla u_\varepsilon| \psi^2\,dx\\&\quad\quad+\theta\beta \int_\Omega \frac{\psi ^2}{\left(a_\varepsilon(|\nabla u_\varepsilon|)\right)^\beta}\,dx+s_{a_\varepsilon}^2 \int_\Omega \left(a_\varepsilon(|\nabla u_\varepsilon|)\right)^{2-\beta}|D^2u_\varepsilon|^2\,dx\\&\quad\quad+\int_{\partial \Omega} \left(a_\varepsilon (|\nabla u_\varepsilon|)\right)^{1-\beta} (\nabla u_\varepsilon\cdot\nu) \psi^2 \,d\mathcal{H}^{n-1}.
        \end{split}
    \end{equation}
    for any  $\theta>0$, $\beta\in \R$,  for any functions $\psi\in C^{\infty}_c(\R^n)$, and where $\nu$ denotes the outward normal vector on $\partial \Omega$. Let $\psi\in C^{\infty}_c(B_{2\rho}(x))$ a standard cut-off function such that $\psi=1$ in $B_{\rho}(x)$, $\psi=0$ in $\left(B_{2\rho}(x)\right)^c$ and the gradient of $\psi$ satisfies $|\nabla \psi|\leq C/\rho$ in $B_{2\rho}(x) \setminus B_{\rho}(x)$, for some positive constant $C$. By \eqref{eq:regolarizz}, using the fact that $f$ has a sign, and for $\theta$ small enough, the inequality \eqref{preinversodelpesoapplicatoauepsilon} becomes,
    \begin{equation}\label{preinversodelpesoapplicatoauepsilon2}
        \begin{split}
            &\int_{\Omega\cap B_\rho(x)} \frac{1}{\left(a_\varepsilon(|\nabla u_\varepsilon|)\right)^\beta} \,dx \\&\quad\leq C(\rho,\theta,\beta,\tau,s_{a_\varepsilon})\left(\int_\Omega \left(a_\varepsilon(|\nabla u_\varepsilon|)\right)^{1-\beta}|\nabla u_\varepsilon|\,dx\right.\\&\quad\quad\left.+ \int_\Omega \left(a_\varepsilon(|\nabla u_\varepsilon|)\right)^{2-\beta}|D^2u_\varepsilon|^2\,dx+\int_{\partial \Omega} \left(a_\varepsilon (|\nabla u_\varepsilon|)\right)^{1-\beta} |\nabla u_\varepsilon| \,d\mathcal{H}^{n-1}\right)\\&\quad:=\mathcal{J}_1+\mathcal{J}_2+\mathcal{J}_3.
        \end{split}
    \end{equation}
    Now we estimate the right-hand side of the previous inequality.
    We distinguish three cases:\\ 
   \textbf{Case 1: $0\leq i_a<s_a$.}
    
    First, we derive a bound for the term $\mathcal{J}_3.$ 
    By \cite[Theorem $1.1$]{Cma_S} (see also \cite{Cma2}), we recall that:
\begin{equation}\label{linfinitogradiente}
    \|\nabla{u}_\varepsilon\|_{L^{\infty}(\Omega)}\leq C(n,i_a,s_a,|\Omega|,\|\operatorname{tr}\mathcal{B}\|_{L^{n-1,1}(\Omega)},\| f\|_{L(\Omega)}),
\end{equation}
where $C(n,i_a,s_a,|\Omega|,\|\operatorname{tr}\mathcal{B}\|_{L^{n-1,1}(\Omega)},\| f\|_{L(\Omega)})$ is a positive constant.  Observe that this constant is independent of $\varepsilon$, as a consequence of Lemma \ref{lemma_a_eps}. For further details, we refer the reader to the proof of Theorem $1.1$ in \cite{Cma_S}. Since $u_\varepsilon \in C^1(\overline \Omega)$, we have $\|\nabla u_\varepsilon\|_{C^{0}(\overline \Omega)}=\|\nabla u_\varepsilon\|_{L^\infty(\Omega)}$. Moreover, by Lemma \ref{lemma_a_eps}  we recall that:
\begin{equation}\label{stimasua_epsilon}
    C_1(\varepsilon+t^2)^{\frac {s_a}{2}}\leq a_\varepsilon(t)\leq C_2(\varepsilon+t^2)^{\frac {i_a}{2}}\quad \text{for } t<<1,
\end{equation}
for some positive constant $C_1$ and $C_2$.\\
By \eqref{linfinitogradiente} and \eqref{stimasua_epsilon}, for $\beta\leq 1$, we obtain,
\begin{equation}\label{stimaJ_3caso1}
\mathcal{J}_3 \leq  C_2\int_{\partial \Omega} \left(\varepsilon+ |\nabla u_\varepsilon|^2\right)^{\frac{(1-\beta)i_a+1}{2}} \,d\mathcal{H}^{n-1}  \leq C, 
\end{equation}
with $C=C(n,\beta,i_a,s_a,|\Omega|,\|\operatorname{tr}\mathcal{B}\|_{L^{n-1,1}(\Omega)},\| f\|_{L(\Omega)})$.

\noindent On the other hand, if $1\leq \beta\leq (s_a+1)/s_a$, proceeding in a similar way, we have 
\begin{equation}\label{stimaJ_3caso11}
\mathcal{J}_3 \leq  C_1\int_{\partial \Omega} \left(\varepsilon+ |\nabla u_\varepsilon|^2\right)^{\frac{(1-\beta)s_a+1}{2}} \,d\mathcal{H}^{n-1}  \leq C.
\end{equation}
 Since the term $\mathcal{J}_1$ can be handled analogously to the previous case, we turn our attention to the estimation of $\mathcal{J}_2$. By Lemma \ref{lemma_a_eps}, for $\beta\leq 1$ we deduce that:
 \begin{equation}\label{stimaJ_2caso1}
 \begin{split}
   \mathcal{J}_2&= \int_\Omega a_\varepsilon(|\nabla u_\varepsilon|) \left(a_\varepsilon(|\nabla u_\varepsilon|)\right)^{1-\beta}|D^2 u_\varepsilon|^2\,dx \\& \leq C_2\int_\Omega a_\varepsilon(|\nabla u_\varepsilon|) \left(\varepsilon+ |\nabla u_\varepsilon|^2\right)^{\frac{(1-\beta)i_a}{2}}|D^2 u_\varepsilon|^2\,dx \leq C, 
 \end{split}
 \end{equation}
 where $C=C(n,i_a,s_a,\beta,L_\Omega,d_\Omega,K_{\Omega},|\Omega|,\|\operatorname{tr}\mathcal{B}\|_{L^{n-1,1}(\Omega)},\| f\|_{L(\Omega)},\| f\|_{W^{1,1}(\Omega)})$.\\
 For $1<\beta<(s_a+1)/s_a$, in a similar way, we obtain the following estimate for $\mathcal{J}_2$:
 \begin{equation}\label{stimaJ_2caso11}
   \begin{split}
        &\mathcal{J}_2 \leq C_1\int_\Omega a_\varepsilon(|\nabla u_\varepsilon|) \left(\varepsilon+ |\nabla u_\varepsilon|^2\right)^{\frac{(1-\beta)s_a}{2}}|D^2 u_\varepsilon|^2\,dx \leq C.
   \end{split}  
 \end{equation}
 By \eqref{stimaJ_3caso1}, \eqref{stimaJ_3caso11}, \eqref{stimaJ_2caso1} and \eqref{stimaJ_2caso11}, the inequality \eqref{preinversodelpesoapplicatoauepsilon2} becomes 
 \begin{equation*}
     \int_{\Omega\cap B_\rho(x)} \frac{1}{\left(a_\varepsilon(|\nabla u_\varepsilon|)\right)^\beta} \,dx \leq C,
     \end{equation*}
     with $C=C(\rho,\tau,n,i_a,s_a,\beta,L_\Omega,d_\Omega,K_{\Omega},|\Omega|,\|\operatorname{tr}\mathcal{B}\|_{L^{n-1,1}(\Omega)},\| f\|_{L(\Omega)},\| f\|_{W^{1,1}(\Omega)})$.

Since $a_\varepsilon(|\nabla u_\varepsilon|) \rightarrow a(|\nabla u|)$ as $\varepsilon \rightarrow 0$ in $\Omega \setminus Z_u$ and by Remark \ref{misuraZu}, using Fatou's Lemma we obtain:
\begin{equation}\label{finestep1}
     \int_{\Omega\cap B_\rho(x)} \frac{1}{\left(a(|\nabla u|)\right)^\beta} \,dx\leq C,
     \end{equation}
     with $C=C(\rho,\tau,n,i_a,s_a,\beta,L_\Omega,d_\Omega,\Psi_{\Omega},|\Omega|,\|\operatorname{tr}\mathcal{B}\|_{L^{n-1,1}(\Omega)},\| f\|_{L(\Omega)},\| f\|_{W^{1,1}(\Omega)})$.
Notice that the dependence on $\Psi_\Omega$ is due to Lemma \ref{antoninho2}. We remark that the function $\Psi_\Omega$ has a finite value due to Remark \ref{antony2}.

\textbf{Case 2: $i_a<0<s_a$.}

We estimate the term $J_2$ of the inequality \eqref{preinversodelpesoapplicatoauepsilon2}. The other terms can be estimated in a similar way as done in the previous case. By Lemma \ref{lemma_a_eps}, for $(i_a+1)/i_a<\beta\leq 1$ we deduce that:
 \begin{equation}\label{stimaJ_2caso2}
 \begin{split}
   \mathcal{J}_2&= \int_\Omega a_\varepsilon(|\nabla u_\varepsilon|) \left(a_\varepsilon(|\nabla u_\varepsilon|)\right)^{1-\beta}|D^2 u_\varepsilon|^2\,dx \\& \leq C_2\int_\Omega a_\varepsilon(|\nabla u_\varepsilon|) \left(\varepsilon+ |\nabla u_\varepsilon|^2\right)^{\frac{(1-\beta)i_a}{2}}|D^2 u_\varepsilon|^2\,dx \leq C, 
 \end{split}
 \end{equation}
 where $C=C(n,i_a,s_a,\beta,L_\Omega,d_\Omega,K_{\Omega},|\Omega|,\|\operatorname{tr}\mathcal{B}\|_{L^{n-1,1}(\Omega)},\| f\|_{L(\Omega)},\| f\|_{W^{1,1}(\Omega)})$.\\
 For $1<\beta<(s_a+1)/s_a$, in a similar way, we obtain the following: 
 \begin{equation}\label{stimaJ_2caso21}
   \begin{split}
        &\mathcal{J}_2 \leq C_1\int_\Omega a_\varepsilon(|\nabla u_\varepsilon|) \left(\varepsilon+ |\nabla u_\varepsilon|^2\right)^{\frac{(1-\beta)s_a}{2}}|D^2 u_\varepsilon|^2\,dx \leq C.
   \end{split}  
 \end{equation} 
The same arguments as in the previous case lead us to conclude that,
\begin{equation*}
     \int_{\Omega\cap B_\rho(x)} \frac{1}{\left(a(|\nabla u|)\right)^\beta} \,dx \leq C,
     \end{equation*}
     with $C=C(\rho,\tau,n,i_a,s_a,\beta,L_\Omega,d_\Omega,\Psi_{\Omega},|\Omega|,\|\operatorname{tr}\mathcal{B}\|_{L^{n-1,1}(\Omega)},\| f\|_{L(\Omega)},\| f\|_{W^{1,1}(\Omega)})$.

\textbf{Case 3: $i_a<s_a\leq 0$}.

This final case can be treated similarly. Using analogous arguments, for $\\\beta> (i_a+1)/i_a$ we deduce \begin{equation*}
     \int_{\Omega\cap B_\rho(x)} \frac{1}{\left(a(|\nabla u|)\right)^\beta} \,dx \leq C.
     \end{equation*}
\textit{Step $2$.} Now we remove the assumption \eqref{bordoregolare_P}. In this regard, we consider a sequence of open sets $\{\Omega_m\}$ approximating $\Omega$ in the sense of Lemma \ref{approxcap}. For $m\in \N$, let ${u}_m$ be the weak solution to the problem:
\begin{equation*}
    -\operatorname{ div}(a(|\nabla u_m|)\nabla{ u_m})=   { f}_m(x) \quad \text{in } \Omega_m,
\end{equation*}
with either $ u_m= 0$ or $\frac{\partial {u}_m}{\partial {\nu}}=0$ on $\partial \Omega_m$ and ${f}_m$ has the same regularity properties as in \textit{Step $2$} of the proof of Theorem \ref{teo1INTRO}.  In all cases, by inequality \eqref{finestep1} of \textit{Step $1$}, applied to ${u}_m$, 
\begin{equation*}
    \int_{\Omega\cap B_\rho (x)} \frac{1}{\left(a(|\nabla u|)\right)^\beta} \,dx\leq\int_{\Omega_m\cap B_\rho (x)} \frac{1}{\left(a(|\nabla u|)\right)^\beta} \,dx \leq C,
\end{equation*}
where $C=C(n,i_a,s_a,\rho,\beta,\tau,L_\Omega,d_\Omega,\Psi_{\Omega},\|\operatorname{tr}\mathcal{B}\|_{L^{n-1,1}(\Omega)},\| f\|_{L(\Omega)},\| f\|_{W^{1,1}(\Omega)})$ is a positive constant. As in the proof of Theorem \ref{teo1INTRO} (see \textit{Step 2}), the dependence of the constant can be recovered. In addition, for $m$ sufficiently large, the functions $f_m$ preserve their sign, possibly after restricting the neighborhood, due to the properties of the extension operator and the structure of mollifiers.

Moreover, by \textit{Step $2$} of the proof of Theorem \ref{teo1INTRO},
\begin{equation*}
    \nabla {u}_m \rightarrow  \nabla{u} \qquad \text{a.e.} \quad\text{in } \Omega.
\end{equation*}
Using Fatou's Lemma we deduce:
\begin{equation}\label{peso_reg1}
    \int_{\Omega\cap B_\rho (x)} \frac{1}{\left(a(|\nabla u|)\right)^\beta}\,dx \leq C.
\end{equation}

\textit{Step $3$.} In the last step, we remove the assumption \eqref{reg_f_P}. Let $f \in W^{1,1}(\Omega) \cap L(\Omega)$. By standard density argument one can infer that there exists a sequence $\{{f}_k\} \subset C^\infty(\overline \Omega)$ (as in \textit{Step $2$}) such that
\begin{equation}\label{conv_f_P}
    {f}_k \rightarrow {f} \quad \text{in } W^{1,1}(\Omega) \cap L(\Omega).
\end{equation}

Now we consider a sequence $\{{u}_k\}$ of weak solutions to the following problem:
\begin{equation*}
    -\operatorname{ div}(a(|\nabla u_k|)\nabla{ u_k})=   { f}_k(x) \quad \text{in } \Omega,
\end{equation*}
with either Dirichlet or Neumann homogeneous boundary conditions.

By inequality \eqref{peso_reg1} of \textit{Step $2$} applied to $ u_k$ we deduce
\begin{equation*}
    \begin{split}
&\int_{\Omega\cap B_\rho (x)} \frac{1}{\left(a(|\nabla u|)\right)^\beta}\,dx \\&\quad\leq C(n,i_a,s_a,\rho,\beta,\tau,L_\Omega,d_\Omega,\Psi_{\Omega},\|\operatorname{tr}\mathcal{B}\|_{L^{n-1,1}(\Omega)},\| f_k\|_{L(\Omega)},\| f_k\|_{W^{1,1}(\Omega)})\\
&\quad\leq \bar C(n,i_a,s_a,\rho,\beta,\tau,L_\Omega,d_\Omega,\Psi_{\Omega},\|\operatorname{tr}\mathcal{B}\|_{L^{n-1,1}(\Omega)},\|  f\|_{L(\Omega)},\| f\|_{W^{1,1}(\Omega)}),
\end{split}
\end{equation*}
where the last inequality follows by \eqref{conv_f_P}.

Moreover, from \textit{Step $3$} of the proof of Theorem \ref{teo1INTRO} we have
\begin{equation*}
    \nabla{u}_k \rightarrow \nabla{u} \qquad \text{a.e.} \quad\text{in } \Omega.
\end{equation*}
Hence, applying Fatou's Lemma, the thesis follows.

\end{proof}

Let us now consider the particular case where the domain is convex:

\begin{proof}[Proof of Theorem \ref{conv_d_peso}]
    As in the proof of Theorem \ref{conv_d}, the result follows from Theorem \ref{peso_stima_Intro}, with minor modifications. The only difference lies in \textit{Step 2}, where we consider a sequence $\Omega_m$ of bounded convex open sets that approximate $\Omega$ from the exterior in the sense of the Hausdorff distance. We also observe that Theorem \ref{conv_d} allows us to derive the estimate \eqref{stimaJ_2caso11} with a constant independent of $K_\Omega$. The remainder of the argument proceeds as in the proof of Theorem \ref{peso_stima_Intro}.
\end{proof}

\vspace{0.3 cm}

The integrability of ${a(|\nabla u|)}^{-1}$, combined with the second-order estimates, enables us to derive additional information on the second derivatives of the solution. In particular we have the following result. 
\begin{corollary}
     Let $ u$, $\Omega$ and $ f$ as in Theorem \ref{peso_stima_Intro}. If $\inf_{t\in[0,M]} a(t)=0$, for every $M>0$, and $s_a\geq 1$, then $$u\in W^{2,q}(\Omega),\quad  \text{with } 1\leq q< \frac{s_a+1}{s_a}.$$
     
\end{corollary}
\begin{proof}
Given $\varepsilon$ small enough, we consider a weak solution ${u_\varepsilon}$ to the approximating problem:
\begin{equation}\label{eq:regolarizz2}
    -\operatorname{ div}( a_\varepsilon(|\nabla u_\varepsilon|)\nabla{ u_\varepsilon})=f,
\end{equation}
under Dirichlet or Neumann boundary conditions.\\
Since $s_a\geq 1$, using Holder inequality and by the properties of $a_\varepsilon$, we deduce
\begin{equation}\label{cose}
    \begin{split}
        \int_\Omega |D^2u_\varepsilon|^q\,dx &=\int_\Omega (|D^2u_\varepsilon|^2)^\frac{q}{2}\frac{(a_\varepsilon(|\nabla u_\varepsilon|))^{\frac{q}{2}}}{(\varepsilon + |\nabla u_\varepsilon|^2)^{\frac{\alpha q}{4}}}\frac{(\varepsilon + |\nabla u_\varepsilon|^2)^{\frac{\alpha q}{4}}}{(a_\varepsilon(|\nabla u_\varepsilon|))^{\frac{q}{2}}}\,dx \\& \leq \left(\int_\Omega \frac{a_\varepsilon(|\nabla u_\varepsilon|)}{(\varepsilon + |\nabla u_\varepsilon|^2)^{\frac{\alpha}{2}}} |D^2u_\varepsilon|^2 \ dx\right)^{\frac{q}{2}}\left(\int_\Omega\frac{(\varepsilon + |\nabla u_\varepsilon|^2)^{\frac{\alpha q}{2(2-q)}}}{(a_\varepsilon(|\nabla u_\varepsilon|))^{\frac{q}{2-q}}}\,dx\right)^{\frac{2-q}{2}}\\& \leq \left(\int_\Omega \frac{a_\varepsilon(|\nabla u_\varepsilon|)}{(\varepsilon + |\nabla u_\varepsilon|^2)^{\frac{\alpha}{2}}} |D^2u_\varepsilon|^2 \ dx\right)^{\frac{q}{2}}\left(\int_\Omega\frac{1}{(a_\varepsilon(|\nabla u_\varepsilon|))^{\frac{q(s_a-\alpha)}{s_a(2-q)}}}\,dx\right)^{\frac{2-q}{2}}
    \end{split}
\end{equation}
provided that $|\nabla u_\varepsilon|$ is sufficiently small; otherwise, the statement is trivial. For $\alpha\approx 1^+$ and $1 \leq q<(s_a+1)/s_a$, using Theorem \ref{teo1INTRO} and Theorem \ref{peso_stima_Intro}, we obtain the following:
\begin{equation*}
    \int_\Omega |D^2u_\varepsilon|^q\,dx\leq C,
\end{equation*}
where $C=C(i_a,s_a,\beta,\tau,q,\rho,n,L_\Omega,d_\Omega,\|{f}\|_{L(\Omega)},\|{f}\|_{W^{1,1}(\Omega)})$ is a positive constant independent of $\varepsilon$. Now we conclude that $u\in W^{2,q}(\Omega)$. Indeed, by the previous estimate, it follows that there exists ${W} \in W^{1,q}(\Omega)$ such that
\begin{equation*}
     \nabla {u}_{\varepsilon} \rightarrow {W} \text{ in } L^q(\Omega), \quad
           \nabla u_{\varepsilon} \rightharpoonup {W}\text{ in }W^{1,q}(\Omega)\quad \text{as  } \varepsilon \rightarrow 0. 
\end{equation*}
Moreover, by \textit{Step $3$} of the proof of Theorem \ref{teo1INTRO}:
\begin{equation*}
    \nabla {u}_{\varepsilon} \rightarrow \nabla{u} \quad\text{a.e. in } \Omega.
\end{equation*}
Hence,
\begin{equation*}
    \nabla {u}={W} \in W^{1,q}(\Omega).
\end{equation*}

\end{proof}

\begin{center}
{\bf Acknowledgements}
\end{center} 
The authors are supported by PRIN PNRR P2022YFAJH \emph{Linear and Nonlinear PDEs: New directions and applications.} D. Vuono has been partially supported by \emph{INdAM-GNAMPA Project Regularity and qualitative aspects of nonlinear PDEs via variational and non-variational approaches} E5324001950001.

\begin{center}
	{\sc Data availability statement}\
	All data generated or analyzed during this study are included in this published article.
\end{center}

\

\begin{center}
	{\sc Conflict of interest statement}
	\
	The authors declare that they have no competing interest.
\end{center}

\end{document}